\documentclass[11pt]{article}

\usepackage[utf8]{inputenc} 
\usepackage[T1]{fontenc}    
\usepackage{url}            
\usepackage{booktabs}       
\usepackage{nicefrac}       
\usepackage{microtype}      
\usepackage{multirow}
\usepackage[noblocks]{authblk}
\usepackage{amssymb,arydshln}
\usepackage[nodayofweek]{datetime}
\usepackage{float}
\usepackage{bm}

\usepackage[shortlabels]{enumitem}
\setlist[itemize]{itemindent=0ex,itemsep=-0.5ex,leftmargin=3ex,topsep=5pt}
\setlist[enumerate]{label={\arabic*)},itemindent=0ex,itemsep=-0.5ex,leftmargin=3.6ex,topsep=5pt,labelwidth=3.6ex,labelsep=1.5ex}

\usepackage[font=small,labelfont=bf]{caption}
\usepackage{amsmath,amsfonts,nccmath,mathtools,mathrsfs}
\usepackage{tcolorbox}
\usepackage{xcolor}
\usepackage{setspace}
\usepackage[colorlinks = true,
            linkcolor = blue,
            urlcolor  = blue,
            citecolor = blue,
            anchorcolor = blue]{hyperref}

\usepackage[margin=1in]{geometry}

\pgfmathdeclarefunction{cubic}{1}{%
  \pgfmathparse{-2*(#1+2)*(#1+2)*(#1-2)+40}%
}

\pgfmathdeclarefunction{bicuadratic}{1}{%
  \pgfmathparse{(#1-1)*(#1-1)*(#1-1)*(#1-1)+10}%
}

\pgfmathdeclarefunction{cuadratic}{1}{%
    \pgfmathparse{(-(2*#1)^2)+70}%
}

\usepackage{pifont}

\usepackage[numbers,merge,sort&compress]{natbib}
\usepackage{amsthm}
\usepackage{graphicx,color}
\usepackage{subcaption}

\newtheorem{theorem}{Theorem}[section]

\newtheorem{proposition}{Proposition}[section]

\newtheorem{lemma}{Lemma}[section]
\newtheorem{corollary}{Corollary}[section]

\newtheorem{assumption}{Assumption}

\theoremstyle{definition}
\newtheorem{definition}{Definition}[section]
\newtheorem{example}{Example}[section]

\newtheorem{remark}{Remark}[section]

\usepackage[nameinlink,capitalise]{cleveref}
\crefname{equation}{}{}
\crefname{theorem}{Theorem}{Theorems}
\crefname{definition}{Definition}{Definitions}
\crefname{corollary}{Corollary}{Corollaries}
\crefname{example}{Example}{Examples}
\crefname{assumption}{Assumption}{Assumptions}
\crefname{lemma}{Lemma}{Lemmas}
\crefname{proposition}{Proposition}{Propositions}
\crefname{figure}{Figure}{Figures}
\crefname{table}{Table}{Tables}
\crefname{fact}{Fact}{Facts}
\crefname{conjecture}{Conjecture}{Conjectures}
\crefname{section}{Section}{Sections}
\crefname{appendix}{Appendix}{Appendices}
\Crefname{equation}{}{}
\Crefname{theorem}{Theorem}{Theorems}
\Crefname{corollary}{Corollary}{Corollaries}
\Crefname{example}{Example}{Examples}
\Crefname{lemma}{Lemma}{Lemma}
\Crefname{proposition}{Proposition}{Proposition}
\Crefname{figure}{Figure}{Figures}
\Crefname{table}{Table}{Tables}
\Crefname{section}{Section}{Sections}
\Crefname{appendix}{Appendix}{Appendices}

\newcommand{\tr}{{{\mathsf T}}}

\newcommand{\tran}{{{\mathsf T}}}

\newcommand{\mK}{{\mathsf{K}}}

\DeclarePairedDelimiterX\setv[2]{\{}{\}}{#1 \;\delimsize\vert\; #2}
\DeclarePairedDelimiterX\setc[2]{\{}{\}}{#1 : #2}

\makeatletter
\newcommand{\removelatexerror}{\let\@latex@error\@gobble}
\makeatother

\title{\bf No Spurious Local Minima in Full-Order Linear Quadratic Gaussian Control%
\thanks{The work of Yujie Tang is supported by the National Natural Science Foundation of China through Grant 72431001. The work of Yang Zheng is supported by NSF CAREER 2340713.}}
\author[1]{Yang Zheng}
\author[2]{Yujie Tang}

\affil[1]{\small Department of Electrical and Computer Engineering, University of California San Diego}
\affil[2]{\small Department of Control Science and Systems Engineering, Peking University}
\date{} 

\begin{document}

\maketitle
\vspace{-4mm}

\begin{abstract}

This paper studies the nonconvex optimization landscape of Linear Quadratic Gaussian (LQG) control under direct state-space parameterization. Although the LQG cost may possess suboptimal stationary points, whether it admits suboptimal local minima has remained open. We answer this question negatively: Every local minimum of the full-order LQG cost is globally optimal. More generally, at any controller order, every local minimum corresponding to an uncontrollable or unobservable realization attains the globally optimal LQG cost over all stabilizing policies of arbitrary orders. We further show that any suboptimal policy can be augmented with decoupled stable controller states, so that the augmented policy admits a direction of negative curvature with probability one. Consequently, any suboptimal stationary point in full-order LQG cost can be converted into a strict saddle with probability one. One key proof idea connects the state-space Hessian to a frequency-domain global-optimality condition derived from the Youla parameterization. These results reveal a benign local-minimum landscape despite the presence of suboptimal stationary points.

\end{abstract}

\section{Introduction}

Linear quadratic Gaussian (LQG) control is one of the foundational problems in optimal control \cite{zhou1996robust}. This classical problem provides a canonical setting for decision-making under partial observation and stochastic disturbances. 
Under standard assumptions, the elegant separation principle yields a globally optimal LQG controller by combining a linear quadratic regulator (LQR) with a Kalman filter, each obtained from an algebraic Riccati equation \cite[Chapter 14.9]{zhou1996robust}. 
Despite this explicit~solution, the LQG cost is highly nonconvex when expressed directly in the controller parameters~\cite{hyland1984optimal,tang2023analysis,zheng2022escaping}. 

Motivated in part by the successes in reinforcement learning, direct policy optimization has received renewed interest as an alternative approach to controller synthesis \cite{hu2023toward,recht2019tour,TalebiZhengKraislerLiMesbahi2026PolicyOptimization,zheng2026benignI,zheng2026benignII,zheng2025extended}. It searches directly over controller parameters, without resorting to model-based algebraic Riccati equations or transforming the problem into a convex formulation, and is therefore naturally amenable to data-driven and model-free implementations \cite{hu2023toward}. This perspective has revealed a remarkably benign optimization landscape for  LQR: Under mild assumptions, its stationary point is unique and globally optimal; the cost is coercive, smooth over any sublevel set, and satisfies a gradient-dominance property \cite{fazel2018global,mohammadi2021convergence,fatkhullin2021optimizing,watanabe2026revisiting}. These properties are fundamental to establishing convergence guarantees for local search algorithms for solving LQR. The optimization landscape of the LQG problem, however, is considerably richer and more complicated. To handle partial observation, we need to search over the class of dynamic policies, whose state-space realizations are not unique and may further contain uncontrollable or unobservable modes. Consequently, the set of stabilizing dynamic policies can be disconnected, and the LQG cost can possess suboptimal stationary points \cite{tang2023analysis,zheng2022escaping}.

Recent work has begun to clarify this intricate landscape in LQG control \cite{tang2023analysis,zheng2022escaping,zheng2026benignI,zheng2026benignII,zheng2025extended,li2025policy,kraisler2024output,keivan2024case}. In particular, the set of full-order stabilizing policies is shown to have at most two path-connected components, which are related by a similarity transformation and are equivalent in the frequency domain \cite{tang2023analysis}. Moreover, every stationary point corresponding to a full-order controllable and observable (a.k.a, minimal) controller is globally optimal~\cite{tang2023analysis}. Consequently, any suboptimal stationary point must be a non-minimal controller.  
More recently, a unified framework of extended convex lifting was proposed in \cite{zheng2026benignI,zheng2026benignII,zheng2025extended}, which establishes global optimality for a broader class of \textit{nondegenerate} stationary policies, covering certain non-minimal controllers. In \cite{li2025policy}, a frequency-domain analysis exploits the convexity enabled by Youla parameterization and derives a necessary and sufficient condition for global optimality in LQG control. 
Nevertheless, these results do not characterize all degenerate, non-minimal stationary points. In particular, all previously known suboptimal stationary points are (strict or high-order) saddles, and existing theory does not rule out the possibility that a full-order LQG controller is a suboptimal local minimum. This leaves a fundamental question open: 

\vspace{-1mm}
\begin{center}
\emph{Can the full-order LQG cost possess a suboptimal local minimum?}
\end{center}
\vspace{-1mm}

\noindent This question is important from both geometric and algorithmic perspectives. A suboptimal~local minimum would trap local search algorithms, since no sufficiently small perturbation could decrease the cost. In contrast, saddle points can often be escaped through suitable perturbations or reparameterizations \cite{zheng2022escaping}. Resolving this question therefore determines whether the nonconvex landscape of full-order LQG control contains genuinely bad local minima or whether all suboptimal stationary points are necessarily saddles that may be escaped algorithmically.  

\subsection{Our contributions}
In this paper, we answer the question above negatively and further characterize the role of non-minimal controller realizations in the LQG landscape. Our main contributions are summarized~as~follows.

\begin{itemize}
    \item  \textbf{No spurious local minima in full-order LQG control.} 
We prove that, for any controller order $q$, every local minimum of $J_q$ corresponding to an uncontrollable or unobservable realization is globally optimal over stabilizing controllers of all possible orders (\cref{theorem:nonminimal-local-minimum-global}). This result is stronger than global optimality within the fixed-order policy class. Combining it with the known global optimality of controllable and observable stationary points \cite{tang2023analysis}, we conclude that every local minimum of the full-order LQG cost $J_n$ is globally optimal (\Cref{corollary:local-minimum-J-n}). We also provide reduced-order examples showing that the full-order condition and the non-minimality condition in these results cannot be removed in general (see \Cref{example:LQG-1,example:LQG-2}).

\item \textbf{Negative curvature through controller augmentation.} Given any suboptimal policy $\mK\in\mathcal{C}_q$, where $\mathcal{C}_q$ denotes the set of stabilizing policies of order $q$, we can augment it with a decoupled stable mode of eigenvalue $\lambda<0$, which does not change its input-output behavior or the LQG cost. We prove that, for almost every $\lambda<0$, the augmented policy $\mK\oplus\lambda I_p$ of order $q + p$ remains stabilizing and admits directions of both strictly negative and strictly positive curvatures for the LQG cost $J_{q+p}$ (\Cref{theorem:hessian-negative-direction}). In particular, if $\mK$ is stationary, then $\mK\oplus\lambda I_p$ is a strict saddle for almost every $\lambda<0$. 
In other words, any stationary point of the full-order LQG cost is either globally optimal or is convertible into a strict saddle with probability one (\Cref{corollary:stationary-point-in-LQG-control}). 
This result significantly strengthens the structured policy transformation approach in \cite{zheng2022escaping}.

\item \textbf{A connection between state-space curvature and frequency-domain optimality.} Our proofs rely on a new connection between the Hessian of the state-space LQG cost and a frequency-domain global-optimality condition. For a stabilizing policy $\mK$ of arbitrary order $q$, we construct a transfer matrix $\mathbf{R}_{\mK}(s)$ and prove in \Cref{theorem:global_optimality_RK} that 
\[ 
\mK \text{ is globally optimal} \quad\Longleftrightarrow\quad \mathbf{R}_{\mK}(s)\equiv 0. 
\]
This characterization is sufficient and necessary and applies to controllers of arbitrary order, which greatly strengthens the characterization in \cite{zheng2022escaping}. This result also  appeared~in~\cite{li2025policy}. We remove~the controller-order condition imposed in \cite{li2025policy} and present a more direct proof.~One key proof idea is to utilize the convex reformulation of the LQG problem using Youla parameterization in the frequency domain. We further derive the explicit Hessian identity 
$
D^2J_{q+p}\!\left(\mK\oplus\lambda I_p\right) [\Delta,\Delta] = 4\operatorname{tr}\!\left( \Delta_{21}\mathbf{R}_{\mK}(-\lambda)^\top\Delta_{12} \right) 
$  
for a suitable class of directions $\Delta$ (see \Cref{lemma:hessian_augmented_scalar}). Combining this identity with the Kalman canonical decomposition allows local second-order information in the state-space parameterization to certify global optimality in the frequency domain, i.e., $\mathbf{R}_{\mK}(s)\equiv 0$.
\end{itemize}

\subsection{Related work}

\noindent\textbf{Policy optimization and landscape analysis for linear control.}
Direct policy optimization has received renewed interest as a flexible approach to controller synthesis that is naturally compatible with data-driven and model-free settings \cite{hu2023toward,recht2019tour,TalebiZhengKraislerLiMesbahi2026PolicyOptimization,zheng2026benignI,zheng2026benignII,zheng2025extended}.  
Most theoretical~developments have focused on the setting with full-state feedback. In particular, both continuous-time and discrete-time LQR problems are shown to enjoy the gradient dominance property, and thus local search algorithms can converge to the optimal LQR controller under mild conditions \cite{fazel2018global,mohammadi2021convergence,fatkhullin2021optimizing}. A unified gradient-dominance analysis has recently been established in \cite{watanabe2026revisiting,watanabe2026gradient}.  
Beyond the classical LQR, related landscape and convergence results have been established for risk-sensitive control \cite{zhang2021policy,pai2026policy}, LQ differential game \cite{zhang2019policy,watanabe2025semidefinite},  Markov jump LQ control \cite{jansch2022policy}, Kalman filtering \cite{zhang2023learning,qian2025model}, and robust control \cite{guo2022global,watanabe2026policy,wang2026zeroth}. We refer to \cite{hu2023toward,TalebiZhengKraislerLiMesbahi2026PolicyOptimization} for two comprehensive surveys.

When we only have partial observation, the landscape of policy optimization becomes richer and more involved, where optimal policies are generally dynamic. Classical work derived first-order optimality conditions for fixed-order dynamic controllers through coupled optimal-projection equations \cite{hyland1984optimal}, but these conditions do not characterize global optimality. More recent landscape analyses showed that the set of full-order stabilizing policies has at most two path-connected components, and that all minimal stationary points for LQG control are globally optimal \cite{tang2023analysis}; see also \cite{zheng2026benignI,zheng2026benignII,zheng2025extended} for further global guarantees on non-degenerate stationary points. Therefore, suboptimal stationary points must be non-minimal controllers. Some of these points are in fact high-order saddles whose Hessians vanish, which motivates structured perturbation schemes that convert them into strict saddles \cite{zheng2022escaping}. A related study exploits the similarity symmetry through a quotient-manifold formulation, leading to local convergence guarantees over minimal LQG controllers \cite{kraisler2024output}.
 
In this work, we characterize local minima of LQG control without imposing minimality or other nondegeneracy conditions, and prove that the full-order LQG landscape has no spurious local~minima.

\medskip

\noindent\textbf{Hidden convexity and controller parameterizations.} It is known that many classical control problems are nonconvex, but they often enjoy hidden convexity \cite{boyd1991linear}. 
Classical solutions to LQG control exploit hidden convexity through Riccati equations \cite{zhou1996robust} or suitable changes of variables leading to linear matrix inequalities (LMIs) \cite{boyd1994linear,gahinet1994linear,scherer1997multiobjective}. In the frequency domain, the Youla parameterization represents all stabilizing controllers through a stable transfer matrix and converts the LQG problem into a convex $\mathcal{H}_2$ optimization problem \cite{youla1976modern}. Alternative convex representations can be established using system-level and input-output parameterizations \cite{wang2019system,furieri2019input,zheng2022system}, all of which use a notion of closed-loop convexity \cite{zheng2021equivalence,boyd1991linear}. Although these parameterizations enable global controller synthesis, they optimize over transformed variables and therefore do not directly characterize the local geometry of the original state-space policy parameters. 
More recently, the extended convex lifting (\texttt{ECL}) framework established global optimality for a broad class of nondegenerate stationary LQG policies \cite{zheng2026benignI,zheng2026benignII,zheng2025extended}. The key idea of \texttt{ECL} is to connect the nonconvex state-space formulation with an LMI-based convex reformulation, which identifies the underlying sources that enable benign nonconvex landscapes in a range of control problems. In \cite{li2025policy}, a complementary frequency-domain analysis derives a necessary and sufficient global-optimality condition and proposes a globally convergent algorithm in the frequency domain. We finally note that an alternative input-output history parameterization further provides new insights into LQG control \cite{sadamoto2025policy}. 

One key step in our proofs is motivated by \cite{li2025policy}, and we develop a sharper connection between the frequency-domain and state-space viewpoints. In particular, we obtain an order-independent global optimality condition using a centered Youla parameterization. We also establish an explicit identity connecting its optimality residual to the Hessian of the LQG cost at augmented state-space policies. This exact connection allows hidden convexity in the frequency domain to characterize local minima and negative curvature in the original policy space.

\subsection{Paper outline}

The remainder of this paper is organized as follows. We present the problem formulation of LQG control in \Cref{section:preliminaries}. The main results are provided in \Cref{section:main-results}, and the technical proofs are given in \Cref{section:proofs}. \Cref{section:conclusion} concludes this paper. The appendix provides additional technical proofs on Youla parameterization and Hessian computation. 

\vspace{3pt}

\noindent \textbf{Notation.} The set of positive integers is denoted by $\mathbb{Z}_{>0}$, and $\mathbb{Z}_{\geq 0}\coloneq \mathbb{Z}_{>0}\cup\{0\}$. The set of $q\times q$ real Hurwitz stable matrices is denoted by $\mathfrak{H}_q$. The set of $q\times q$ real symmetric matrices is denoted by $\mathbb{S}^q$, and $\mathbb{S}^q_+,\mathbb{S}^q_{++}\subseteq\mathbb{S}^q$ are the sets of positive semidefinite and positive definite matrices, respectively. The set of all $q\times q$ real invertible matrices is denoted by $\mathrm{GL}_q$. $0_{p\times q}$ denotes the ${p\times q}$ zero matrix, and $I_p$ denotes the $p\times p$ identity matrix; their subscripts may be omitted when they can be inferred from the context. The block-diagonal matrix formed by $A_1,\ldots,A_k$ is denoted by $\operatorname{diag}(A_1,\ldots,A_k)$. The set of all eigenvalues of a square matrix $M$ is denoted by $\operatorname{eig}(M)$. The Hessian of a twice differentiable real-valued function $f$ at $x$ is denoted by $D^2 f(x)$ and is to be viewed as a symmetric bilinear form.

\vspace{3pt}

\section{LQG Control and Problem Formulation} \label{section:preliminaries}

Consider the LTI dynamic system
\begin{equation}
\label{eq:plant}
\begin{aligned}
\dot{x}(t) ={} & Ax(t) + Bu(t) + W^{1/2}w(t), \\
y(t) ={} & Cx(t) + V^{1/2}v(t),
\end{aligned}
\end{equation}
where $x(t)\in \mathbb{R}^n$ is the system state, $u(t)\in \mathbb{R}^m$ is the control input, $y(t)\in \mathbb{R}^\ell$ is the output measurement, and $w \in \mathbb{R}^n,v \in \mathbb{R}^\ell$ are white Gaussian noises with identity intensity matrices, i.e., $\mathbb{E}[w(t)w(\tau)^\tr]=\delta(t-\tau)I_n$ and $\mathbb{E}[v(t)v(\tau)^\tr]=\delta(t-\tau)I_{\ell}$. We assume $w$ and $v$ are mutually independent. 
In \Cref{eq:plant}, we have $A\in\mathbb{R}^{n\times n}$, $B\in\mathbb{R}^{n\times m}$,  $C\in\mathbb{R}^{\ell\times n}$, $W \in \mathbb{S}^{n}_{+}$ and $V \in \mathbb{S}^{\ell}_{++}$. 
We consider the following performance signal 
\begin{equation} \label{eq:performance-signal}
    z(t) = \begin{bmatrix}
Q^{1/2}x(t) \\ R^{1/2} u(t)
\end{bmatrix},
\end{equation}
where $Q \in \mathbb{S}^n_+$ and $R\in \mathbb{S}^m_{++}$ are performance weight matrices. We aim to design a feedback policy to minimize the linear quadratic performance criterion
\begin{equation}
\label{eq:lqg-cost-original}
\lim_{T\to+\infty}\frac{1}{T} 
\mathbb{E}\!\left[\int_0^T
z(t)^\tr z(t)\,d t \right]= \lim_{T\to+\infty}\frac{1}{T}
\mathbb{E}\!\left[\int_0^T
\left(x(t)^\tr Qx(t)+u(t)^\tr Ru(t)\right)d t\right].
\end{equation}
For this goal, it is standard and also sufficient (see, e.g., \cite[Chapter 14]{zhou1996robust}) to consider a dynamic controller of the form
\begin{equation}
\label{eq:dynamic_controller}
\begin{aligned}
\dot{\xi}(t) ={} & A_{\mK}\xi(t)+B_\mK y(t), \\
u(t) ={} & C_\mK \xi(t)
\end{aligned}
\end{equation}
where $\xi(t) \in \mathbb{R}^q$ is the controller internal state; $A_{\mK}\in\mathbb{R}^{q\times q}$, $B_{\mK}\in\mathbb{R}^{q\times\ell}$ and $C_{\mK}\in\mathbb{R}^{m\times q}$ specify the controller dynamics. For notational simplicity, we parameterize the dynamic controller \cref{eq:dynamic_controller} by 
\[
\mK = \begin{bmatrix}
0_{m\times\ell} & C_{\mK} \\
B_{\mK} & A_{\mK}
\end{bmatrix} \in \mathbb{R}^{(m+q)\times (\ell + q)},
\]
and simply refer to $\mK$ as the \emph{policy} corresponding to the dynamic controller~\cref{eq:dynamic_controller}. The number $q$ is called the \emph{order} of the policy $\mK$, meaning the dimension of the controller internal state.\footnote{
The case $q=0$ is allowed and is understood to be the case where one sets $u(t)=0$ for all $t$; the corresponding policy is $\mK=0_{m\times\ell}$.
} If $q = n$, namely, the controller state and the system state have the same dimension, we say the policy $\mK$ is \textit{full-order}. It is a reduced-order policy if $q < n$. 

Substituting the policy \cref{eq:dynamic_controller} into the system \cref{eq:plant,eq:performance-signal},  the closed-loop system is given by
\begin{subequations}
\begin{equation} \label{eq:closed-loop-system}
\begin{aligned}
\begin{bmatrix}
\dot{x}(t) \\ \dot{\xi}(t)
\end{bmatrix}
= {} & A_{\mathrm{cl},\mK} \begin{bmatrix}
x(t) \\ \xi(t)
\end{bmatrix}
+ B_{\mathrm{cl},\mK}
\begin{bmatrix}
w(t) \\ v(t)
\end{bmatrix}, \\
z(t) ={} &  C_{\mathrm{cl},\mK} \begin{bmatrix}
x(t) \\ \xi(t)
\end{bmatrix},
\end{aligned}
\end{equation}
where we denote the closed-loop matrices as 
\begin{equation} \label{eq:closed-loop-matrices}
A_{\mathrm{cl},\mK}\coloneq \begin{bmatrix}
A & BC_\mK \\ B_\mK C & A_\mK
\end{bmatrix},
\ \ 
B_{\mathrm{cl},\mK}\coloneq \begin{bmatrix}
W^{1/2} & 0 \\ 0 & B_\mK V^{1/2}
\end{bmatrix},
\ \ 
C_{\mathrm{cl},\mK}\coloneq \begin{bmatrix}
Q^{1/2} & 0 \\
0 & R^{1/2}C_\mK
\end{bmatrix}.
\end{equation}   
\end{subequations}
Let
\[
\mathcal{V}_q\coloneq
\setc*{\mK = \begin{bmatrix}
0_{m\times\ell} & C_{\mK} \\
B_{\mK} & A_{\mK}
\end{bmatrix}}{A_{\mK}\in\mathbb{R}^{q\times q}, B_{\mK}\in\mathbb{R}^{q\times\ell}, C_{\mK}\in\mathbb{R}^{m\times q}}
\]
denote the linear space of all policies of order $q$. We say that $\mK\in\mathcal{V}_q$ internally stabilizes the plant~\cref{eq:plant} if $A_{\mathrm{cl},\mK}$ in \cref{eq:closed-loop-matrices} is Hurwitz stable (i.e., all its eigenvalues have negative real parts), and let
\[
\mathcal{C}_q \coloneq \setc*{\mK\in\mathcal{V}_q}{A_{\mathrm{cl},\mK}\text{ is Hurwitz stable}}
\]
denote the set of all internally stabilizing policies of order $q$.

From \cref{eq:closed-loop-system}, the closed-loop transfer function from $d=\begin{bmatrix} w^\tran & \!\!\!v^\tran\end{bmatrix}^\tran$ to $z$ is given by
\[
\mathbf{T}_{zd,\mK}(s) = C_{\mathrm{cl},\mK}(sI-A_{\mathrm{cl},\mK})^{-1} B_{\mathrm{cl},\mK}.
\]
For any $q\in\mathbb{Z}_{\geq 0}$ such that $\mathcal{C}_q$ is non-empty, we define the function $J_q:\mathcal{C}_q\rightarrow\mathbb{R}$ by the squared $\mathcal{H}_2$ norm of $\mathbf{T}_{zd,\mK}$ as 
$
J_q(\mK) \coloneq \left\|\mathbf{T}_{zd,\mK}\right\|_{\mathcal{H}_2}^2,\,\forall\mK\in\mathcal{C}_q.
$ 
It is a standard result in control theory that $J_q(\mK)$ equals the linear quadratic performance \cref{eq:lqg-cost-original} under the policy $\mK\in\mathcal{C}_q$. 

The problem of LQG control is formally given by 
\begin{equation} \label{eq:LQG-control}
    J^\star \coloneq \inf_{q\in\mathbb{Z}_{\geq 0}}\inf_{\mK\in\mathcal{C}_q} J_q(\mK).
\end{equation}
In other words, we aim to find the best dynamic policy \cref{eq:dynamic_controller} of arbitrary order that minimizes the linear quadratic performance \cref{eq:lqg-cost-original}. 
A policy $\mK\in\mathcal{C}_q$ is called \emph{globally optimal} if $J_q(\mK)=J^\star$. 
Under standard assumptions on the plant~\cref{eq:plant,eq:performance-signal}, it is known that there exists a globally optimal policy $\mK^\star$ of order $n$, which can be found by solving two Riccati equations; see e.g., {\cite[Theorem 14.7]{zhou1996robust}}.

\begin{assumption} \label{assuption:standard}
    We assume $R\succ0, V\succ0,\ Q\succeq0,\ W\succeq0$; $(A,B)$ and $(A,W^{1/2})$ are controllable; $(C,A)$ and $(Q^{1/2},A)$ are observable. 
\end{assumption}

\begin{theorem}[{\cite[Theorem 14.7]{zhou1996robust}}] \label{theorem:seperation-controller}
    With \cref{assuption:standard}, a globally optimal controller to \cref{eq:LQG-control} is~given~by 
    \begin{equation}
\label{eq:separation-controller}
A_{\mK,\rm sep}
=
A-BF-LC,\qquad
B_{\mK,\rm sep}=L,\qquad
C_{\mK,\rm sep}=-F,
\end{equation}
where $F=R^{-1}B^\tr S, L=PC^\tr V^{-1}$ with $P$ and $S$ being the unique positive semidefinite solutions to the following Riccati equations
\begin{subequations}
\begin{align}
\label{eq:filter-ARE}
AP+PA^\tr-PC^\tr V^{-1}CP+W&=0, \\
\label{eq:control-ARE}
A^\tr S+SA-SBR^{-1}B^\tr S+Q&=0. 
\end{align}
\end{subequations}
\end{theorem}
The optimal controller \cref{eq:separation-controller} admits an observer-based form and represents the celebrated separation principle: It can be viewed as a combination of the optimal LQR gain $F$ and optimal Kalman filter gain $L$. The controller state $\xi$ is the optimal Kalman estimation of the system state $x$.  

Consequently, it suffices to optimize over full-order dynamic policies:
$
    J^\star=\min_{\mK\in\mathcal{C}_n}J_n(\mK).
$ 
This full-order LQG cost function $J_n$ is still highly nonconvex as its domain may be disconnected,  and there may exist suboptimal saddle points \cite[Theorem 5]{tang2023analysis}, \cite[Example 5]{zheng2026benignI}. It is known that the stabilizing set of full-order dynamic policies $\mathcal{C}_n$ has at most two path-connected components \cite[Theorem 1]{tang2023analysis}. Furthermore, for a controllable and observable policy $\mK \in \mathcal{C}_n$, if it is a stationary point, i.e., $\nabla J_n(\mK) = 0$, then it must be globally optimal \cite[Theorem 6]{tang2023analysis}. Consequently, any suboptimal stationary point of $J_n$ must be uncontrollable or unobservable in $\mathcal{C}_n$. 

In the remainder of the paper, we allow arbitrary controller orders and study the local minima and stationary points of $J_q$ over $\mathcal{C}_q$. The full-order case $q=n$ will be an important special case of our main results.

\section{Main Results} \label{section:main-results}

In this section, we summarize our main technical results, and their proofs are postponed to \cref{section:global-optimality} after introducing suitable techniques in \cref{section:Hessian,section:local-minima-Canonical-form}.   

\subsection{Global optimality}

Recall that a policy $\mK\in\mathcal{C}_q$ is called controllable if $(A_\mK,B_\mK)$ is controllable, and observable if $(C_\mK,A_\mK)$ is observable. A policy that is both controllable and observable is also called \textit{minimal}.\footnote{This is the standard systems-theoretic notion of minimality: The controller realization \cref{eq:dynamic_controller} has the smallest possible internal-state dimension among all realizations with the same input-output behavior; see \cite[Theorem 3.16]{zhou1996robust}. It should not be confused with local or global minima in the LQG problem \cref{eq:LQG-control}.}

Our first result provides a global optimality characterization for local minima of the LQG cost $J_q$ that corresponds to an uncontrollable or unobservable policy. Its proof is postponed to \cref{section:global-optimality}.

\begin{theorem} \label{theorem:nonminimal-local-minimum-global}
Let $q\in\mathbb{Z}_{>0}$ satisfy $\mathcal{C}_q\neq\varnothing$, and suppose $\mK\in\mathcal{C}_q$ is a local minimum of $J_q$ which is uncontrollable or unobservable. Then $\mK$ is a globally optimal policy, i.e.,
\[
J_q(\mK) =J^\star = \min_{\mK'\in\mathcal{C}_n}J_n(\mK').
\]
\end{theorem}

This result guarantees that any uncontrollable or unobservable local minimum of $J_q$ not only globally minimizes $J_q$ over $\mathcal{C}_q$, but also attains the globally optimal LQG cost $J^\star$ over all possible stabilizing dynamic policies; see \cref{eq:LQG-control}. We remark that the controller order $q$ can be different from the state dimension $n$. In fact, a reduced order $q < n$ does not play an important role in \Cref{theorem:nonminimal-local-minimum-global}.
To interpret this result, let $\mK\in\mathcal{C}_q$ be uncontrollable or unobservable. Its realization is then not minimal and contains at least one internal controller-state direction that is hidden from the controller input-output map. Such a hidden state does not affect the closed-loop behavior at $\mK$. If $\mK$ is further a local minimum of $J_q$, none of these hidden directions can be used to reduce the LQG cost locally. \Cref{theorem:nonminimal-local-minimum-global} shows that this local property has a \textit{striking} global consequence: no stabilizing dynamic policy of any order can achieve a lower LQG cost; in other words,  $J_q(\mK)=J^\star$.

This interpretation can also be used to certify global optimality of a local minimum that may correspond to a minimal policy in $\mathcal{C}_q$. The idea is to introduce an extra controller state that is uncontrollable and unobservable, and then to test whether the augmented policy admits any local descent direction. Concretely, for any 
$\mK
=
\begin{bmatrix}
0 & C_{\mK} \\
B_{\mK} & A_{\mK}
\end{bmatrix}
\in\mathcal{C}_q$ and $\Lambda\in\mathbb{R}^{p\times p}$, we denote an augmented controller
\begin{equation}
\label{eq:augmented-policy}
\mK\oplus\Lambda\coloneqq
\left[
\begin{array}{c:cc}
0 &C_{\mK} & 0\\ \hdashline
B_{\mK} & A_{\mK} & 0 \\
0 & 0 & \Lambda
\end{array}
\right].
\end{equation}
which corresponds to a dynamic policy of order $q+p$. Note that the appended controller subsystem in \cref{eq:augmented-policy} is completely decoupled and therefore does not affect the controller input-output map or the LQG cost. It is not difficult to see that $\mK\oplus\Lambda \in \mathcal{C}_{q+p}$ if $\Lambda$ is Hurwitz stable. Applying \Cref{theorem:nonminimal-local-minimum-global} to this augmented policy yields the following corollary. 

\begin{corollary} \label{corollary:local-minimum-global}
Let $q\in\mathbb{Z}_{>0}$ satisfy $\mathcal{C}_q\neq\varnothing$, and let $\lambda<0$ be arbitrary. Then $\mK \in\mathcal{C}_{q}$ is a globally optimal policy, i.e., $J_q(\mK) = J^\star$, whenever $\mK\oplus\lambda$ is a local minimum of $J_{q+1}$.
\end{corollary}

\begin{proof}
    It is easy to verify that $\mK\oplus\lambda \in \mathcal{C}_{q+1}$ whenever $\mK \in \mathcal{C}_q$ and $\lambda < 0$. Also, $\mK\oplus\lambda$ is uncontrollable and unobservable. If $\mK\oplus\lambda$ is a local minimum of $J_{q+1}$, \Cref{theorem:nonminimal-local-minimum-global} guarantees that $J_{q+1}(\mK\oplus\lambda) = J^\star$. We thus have 
    $
    J_q(\mK) = J_{q+1}(\mK\oplus\lambda) = J^\star 
    $ by construction.
\end{proof}

By combining \cite[Theorem 6]{tang2023analysis} with \Cref{theorem:nonminimal-local-minimum-global}, we derive the following theorem. 

\begin{theorem} \label{corollary:local-minimum-J-n}
Let $\mK\in\mathcal{C}_n$ be a local minimum of $J_n$, where we recall that $n$ is the state dimension of the plant~\cref{eq:plant}. Then $\mK$ is a globally optimal policy.
\end{theorem}

\begin{proof}
    Let $\mK\in\mathcal{C}_n$ be a local minimum of $J_n$. If $\mK$ is a controllable and observable policy, then \cite[Theorem 6]{tang2023analysis} guarantees that it must be globally optimal. If instead $\mK$ is uncontrollable or unobservable, the result follows directly from  \Cref{theorem:nonminimal-local-minimum-global}. 
\end{proof}

Note that \Cref{theorem:nonminimal-local-minimum-global,corollary:local-minimum-global} certify the global optimality of a special class of local minima $\mK \in \mathcal{C}_q$ that corresponds to uncontrollable or unobservable policies without knowing the system state dimension $n$. Instead, \Cref{corollary:local-minimum-J-n} uses the state dimension $n$ to remove the explicit assumption of controllability or observability. We present two examples to illustrate these results. 

\vspace{-1mm}

\begin{example} \label{example:LQG-1}
    Consider the second-order LQG instance with problem data
    \begin{equation}
\label{eq:LQG-example-data}
    A=\begin{bmatrix}0&1\\-1&0\end{bmatrix},
    \qquad
    B=\begin{bmatrix}1\\0\end{bmatrix},
    \qquad
    C=\begin{bmatrix}1&0\end{bmatrix},
    \qquad
    W=Q=I_2,
    \qquad
    V=R=1.
\end{equation}
We restrict attention to the class of first-order dynamical policies of the form 
\begin{equation} \label{eq:first-order-policy}
    \dot\xi(t)=a\xi(t)+b y(t), \,  u=c\xi(t), 
\end{equation}
where $a,b,c \in \mathbb{R}$. From \cref{eq:closed-loop-matrices}, 
the closed-loop characteristic polynomial is $p(\lambda)=\det(\lambda I-A_{\mathrm{cl},\mK}) = \lambda^3 - a \lambda^2 + (1-bc)\lambda - a$. By the Routh--Hurwitz criterion, the closed-loop system is stable if
and only if $a < 0, bc<0$. Thus,  the stabilizing region  \(\mathcal{C}_1\)  is convex and hence path-connected in the reduced coordinates \((a,bc)\).  In the original coordinates
\((a,b,c)\), \(\mathcal{C}_1\) has two path-connected
components, corresponding to \(b>0,c<0\) and \(b<0,c>0\). 
The LQG cost can be explicitly computed as 
$$
J_1(\mK) =  \frac{4a^2 + 4 + 5b^2c^2 -b^3c^3}{2abc}, \qquad \forall a < 0, bc<0. 
$$
A direct calculation shows that \(J_1\) admits the stationary point $\mK_1 = \begin{bmatrix}
      0 & -2\\1 & -2\sqrt{2}
  \end{bmatrix}$. 
  In fact, \(\mK_1\) is globally optimal over \(\mathcal{C}_1\), and is
therefore a local minimizer within the class of first-order policies.
Moreover, its controller realization is minimal, and the corresponding LQG cost is  $J_1(\mK_1)=4\sqrt{2}$. The cost landscape over \(\mathcal{C}_1\) is illustrated in \Cref{fig:LQG-1-cost}. By contrast, the globally optimal second-order policy given by \cref{theorem:seperation-controller} achieves
$
J^\star
=10(\sqrt{2}-1)\sqrt{2\sqrt{2}-1}
<4\sqrt{2}.
$ 
Thus, \(\mK_1\) is globally optimal within \(\mathcal{C}_1\), but is strictly suboptimal among dynamic policies of arbitrary order for LQG~\cref{eq:LQG-control}. 
This example shows that the uncontrollability or unobservability condition in
\cref{theorem:nonminimal-local-minimum-global} is essential. For any \(\lambda<0\),
\Cref{corollary:local-minimum-global} implies that the augmented policy
$
\mK_1\oplus\lambda\in\mathcal{C}_2
$ 
cannot be a local minimum over \(\mathcal{C}_2\). Thus, there exists a local perturbation on $\mK_1\oplus\lambda$ within \(\mathcal{C}_2\) that strictly decreases the LQG~cost.
\hfill $\square$
\end{example}

\begin{figure}[t]
    \centering
    \begin{subfigure}[t]{0.42\linewidth}
        \centering
        \includegraphics[width=\linewidth]{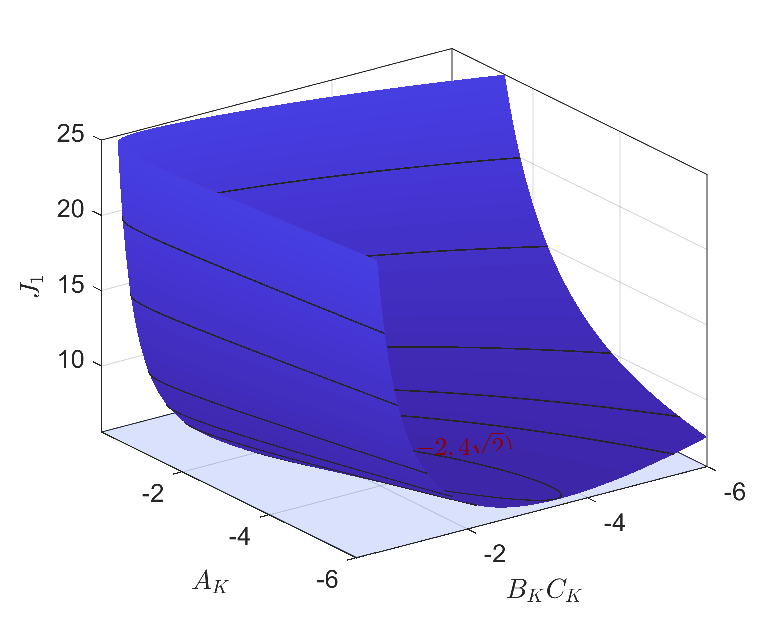}
        \caption{LQG cost in \cref{example:LQG-1}.}
        \label{fig:LQG-1-cost}
    \end{subfigure}
    \hspace{10mm}
    \begin{subfigure}[t]{0.36\linewidth}
        \centering
        \includegraphics[width=\linewidth]{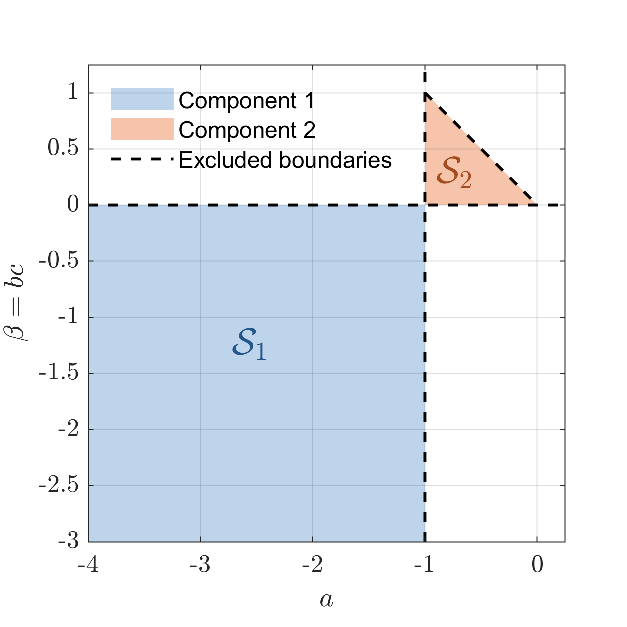}
        \caption{Feasible region $\mathcal{C}_1$ in
        \cref{example:LQG-2}.}
        \label{fig:LQG-2-region}
    \end{subfigure}
    \vspace{-1mm}
    \caption{Illustrations for \cref{example:LQG-1,example:LQG-2}.}
    \label{fig:LQG-examples}
\end{figure}

\vspace{-1mm}
\begin{example} \label{example:LQG-2}
Consider another second-order LQG instance obtained from
\cref{eq:LQG-example-data} by replacing the input matrix with
$
B=\begin{bmatrix}1&1\end{bmatrix}^\tr.
$ 
For the class of first-order dynamic policies in
\cref{eq:first-order-policy}, the closed-loop characteristic polynomial
is
$
p(\lambda) =\det\!\left(\lambda I-A_{\mathrm{cl},\mK}\right) =\lambda^3-a\lambda^2+(1-bc)\lambda-(a+bc).
$ 
The Routh--Hurwitz criterion gives the following characterization of the stabilizing region:
\[
\mathcal{C}_1
=
\left\{(a,b,c)\in\mathbb{R}^3:
a<-1,\ bc<0\right\}
\cup
\left\{(a,b,c)\in\mathbb{R}^3:
-1<a<0,\ 0<bc<-a\right\}.
\]
Thus, the stabilizing region $\mathcal{C}_1$ has two connected components in the
reduced coordinates $(a,\beta = bc)$; see \cref{fig:LQG-2-region} for illustration. In the original coordinates \((a,b,c)\), each of the two components further splits according to the signs of \(b\) and \(c\), and indeed, \(\mathcal{C}_1\) has four path-connected components. 

Since \(bc\neq0\) over \(\mathcal{C}_1\), every stabilizing first-order policy must be controllable and observable (i.e., minimal) in \(\mathcal{C}_1\). The corresponding closed-loop realization is also minimal, and consequently, the LQG cost diverges whenever a policy approaches the boundary of $\mathcal{C}_1$ \cite[Theorem~2]{zheng2026benignI}. Indeed, \(J_1\) attains a minimum in the interior of each reduced component. Solving the corresponding first-order optimality conditions
numerically gives
$
(a_-^\star,\beta_-^\star)
\approx(-7.22,-4.10),
J_1(a_-^\star,\beta_-^\star)
\approx7.44$, and $
(a_+^\star,\beta_+^\star)
\approx(-0.30,0.22),
J_1(a_+^\star,\beta_+^\star)
\approx17.97.
$ 
The Hessian is positive definite at both stationary points. Hence, each point is a strict local minimizer in the reduced coordinates $(a,\beta = bc)$. 
It is clear that the policy \(\mK_{1,+}\) corresponding to $(a_+^\star,\beta_+^\star)$ is a spurious local minimizer of \(J_1\). This example shows that there might be spurious local minima within the set of low-order policies even when every feasible controller realization is controllable and observable. 
\hfill $\square$
\end{example}

\subsection{Suboptimal stationary points as strict saddles}

We next study the second-order geometry of suboptimal stationary points in the full-order LQG cost $J_n$. Although these points cannot be local minima of $J_n$, their Hessians may be degenerate and fail to reveal a descent direction; see \cite[Example 5]{tang2023analysis} and \cite[Example 2]{zheng2022escaping}. We prove that a suitable change of controller realization results in an indefinite Hessian with probability one. Consequently, every stationary point of $J_n$ is either globally optimal or can be converted into a strict saddle with probability one.

Here, a stationary policy $\mK \in \mathcal{C}_q$ is called a \emph{strict saddle} if its Hessian is indefinite. Equivalently, for a strict saddle $\mK\in\mathcal{C}_q$, there exist two directions $\Delta_-,\Delta_+ \in \mathcal{V}_q$ such that 
\[ 
D^2 J_q(\mK)[\Delta_-,\Delta_-] < 0 < D^2 J_q(\mK)[\Delta_+,\Delta_+]. 
\] 
Our main result applies to suboptimal policies of arbitrary order.

\begin{theorem} \label{theorem:hessian-negative-direction}
Let $q\in\mathbb{Z}_{\geq 0}$ satisfy $\mathcal{C}_q\neq\varnothing$,
and suppose $\mK\in\mathcal{C}_q$ is not a globally optimal policy. Then for any $p\geq 1$,  the set
\begin{equation}
\label{eq:augmented_scalar_negative_direction}
\setv*{\lambda\in(-\infty,0)}{\begin{aligned} & D^2 J_{q+p}(\mK\oplus\lambda I_p)[\Delta_{-},\Delta_{-}]< 0 < D^2 J_{q+p}(\mK\oplus\lambda I_p)[\Delta_+,\Delta_+] \\ & \text{ for some }\Delta_-,\Delta_{+}\in \mathcal{V}_{q+p} \qquad \end{aligned}}
\end{equation}
has full Lebesgue measure in $(-\infty,0)$.
\end{theorem}

The proof is postponed to \Cref{section:global-optimality}. Its key idea is to connect the Hessian at $\mK\oplus\lambda I_p$ to a frequency-domain optimality residual $\mathbf{R}_{\mK}(s)$. The suboptimality of the policy $\mK$ implies $\mathbf{R}_{\mK}(s)\not\equiv0$, and so rationality guarantees $\mathbf{R}_{\mK}(-\lambda)\neq0$ for almost every $\lambda<0$. A bilinear Hessian identity then gives directions of both positive and negative curvature. \Cref{theorem:hessian-negative-direction} implies that appending a single decoupled stable state exposes directions with different signs of curvature for almost every choice of its eigenvalue, and this augmented policy preserves the controller input-output behavior and the LQG cost.

To apply this result to stationary points, we first note that the policy augmentation~\cref{eq:augmented-policy} also preserves stationarity. More generally, for any $\mK \in \mathcal{C}_q$ and any Hurwitz matrix $\Lambda \in \mathbb{R}^{p \times p}$, we have 
\begin{equation} \label{eq:augmentation-first-derivative} 
D J_{q+p}(\mK \oplus \Lambda) \left[ \begin{pmatrix} \Delta_{11} & \Delta_{12} \\ \Delta_{21} & \Delta_{22} \end{pmatrix} \right] = D J_q(\mK)[\Delta_{11}], 
\end{equation} 
where $\Delta_{11} \in \mathcal{V}_q$ and the other blocks have compatible dimensions. Indeed, perturbing the original controller block reproduces the first-order variation of $J_q$. At the decoupled realization, perturbing either coupling block alone or the appended state matrix does not change the controller transfer matrix, and thus does not contribute to the first-order variation of $J_{q+p}$. Linearity of the directional derivative then gives \cref{eq:augmentation-first-derivative}. Thus, if $\nabla J_q(\mK)=0$, then $\nabla J_{q+p}(\mK \oplus \Lambda)=0$ (also see \cite[Theorem 4]{tang2023analysis}). 

Together with \cref{theorem:hessian-negative-direction}, this shows that any suboptimal stationary policy of order $q$ becomes a strict saddle of $J_{q+p}$ after augmentation by $\lambda I_p$ for almost every $\lambda<0$. For the full-order LQG problem, the conversion can be performed without increasing the controller order since every suboptimal stationary point of $J_n$ is non-minimal and its redundant states can be replaced by decoupled stable modes. We have the following corollary. 

\begin{corollary} \label{corollary:stationary-point-in-LQG-control}
Let $\mK \in \mathcal{C}_n$ satisfy $\nabla J_n(\mK)=0$ and $J_n(\mK)>J^\star$. Let $\mK_{\min}$ be a minimal realization of this policy, and let $r$ denote its order. Then $r<n$, and, for almost every $\lambda<0$, the policy 
\[ \widehat{\mK}_{\lambda} \coloneq \mK_{\min} \oplus (\lambda I_{n-r}) 
\] is a strict saddle of $J_n$. Moreover, it satisfies $J_n(\widehat{\mK}_{\lambda}) = J_n(\mK).$
\end{corollary} 

\begin{proof} Since every controllable and observable stationary point of $J_n$ is globally optimal \cite{tang2023analysis}, the policy $\mK$ must be non-minimal. Hence $r<n$. It is known that $\mK_{\min} \in \mathcal{C}_r$ remains a stationary point of $J_r$ \cite[Theorem 1]{zheng2022escaping}.

For every $\lambda<0$, the policy $\widehat{\mK}_{\lambda}$ belongs to $\mathcal{C}_n$ and has the same controller transfer matrix and cost as $\mK$. By \cref{eq:augmentation-first-derivative}, it is also stationary. Applying \cref{theorem:hessian-negative-direction} to $\mK_{\min}$ shows that, for almost every $\lambda<0$, there exist two directions $\Delta_-,\Delta_+ \in \mathcal{V}_{n}$ such that 
\[
D^2 J_{n}(\mK_{\min}\oplus\lambda I_{n-r})[\Delta_-,\Delta_-] < 0 < D^2 J_{n}(\mK_{\min}\oplus\lambda I_{n-r})[\Delta_+,\Delta_+]. 
\]
Therefore, $\mK_{\min}\oplus\lambda I_{n-r}$ is a strict saddle of $J_n$ for almost every $\lambda<0$. 
\end{proof} 

This result implies that every suboptimal stationary point of $J_n$ can be converted into a strict saddle of $J_n$ with probability one. This can be achieved by simply sampling $\lambda$ from any probability distribution on $(-\infty,0)$ that is absolutely continuous with respect to Lebesgue measure. 
The proof in \cref{corollary:stationary-point-in-LQG-control} consists of taking a minimal realization and then restoring the controller order using decoupled stable states. This structured transformation preserves the controller input-output behavior and the cost, while creating a direction of negative curvature.
This result also applies even when the Hessian at the original stationary point vanishes, generalizing \cite[Theorem 2]{zheng2022escaping}.

Finally, we revisit \Cref{example:LQG-1,example:LQG-2} to illustrate how appending a single stable controller state reveals negative curvature of the LQG cost. 

\begin{example} \label{example:LQG-3}
For a first-order policy with $b=1$ in \cref{eq:first-order-policy}, consider the perturbed second-order policy
\begin{equation} \label{eq:augmented-policy-example}
\widehat{\mK}_{\lambda}(\eta,\zeta)
=\left[
\begin{array}{c:cc}
0 &c& \zeta\\ \hdashline
1 & a & 0 \\
\eta & 0 & \lambda
\end{array}
\right], 
\qquad \lambda<0,
\end{equation}
where $\eta$ and $\zeta$ specify the input and output couplings of the additional controller state. At~$(\eta,\zeta) = (0,0)$, this additional state is decoupled, and $\widehat{\mK}_{\lambda}(0,0)=\mK\oplus\lambda$ preserves the controller input-output map and the LQG cost. For $\lambda=-1$, we visualize the cost difference
$
\Delta J_2(\eta,\zeta)
\coloneq
J_2\bigl(\widehat{\mK}_{-1}(\eta,\zeta)\bigr)
-
J_2\bigl(\widehat{\mK}_{-1}(0,0)\bigr).
$ 
We also vary $\lambda$ and compute the smallest eigenvalue of the full Hessian $D^2J_2(\mK\oplus\lambda)$.

Recall that the first-order policy $\mK_1$ in \cref{example:LQG-1}, with $(a,c)=(-2\sqrt{2},-2)$, globally minimizes $J_1$ over $\mathcal{C}_1$, but satisfies $J_1(\mK_1)=4\sqrt{2}>J^\star$. The left panel of \Cref{fig:example-1-augmentation} shows that, after appending the stable pole $\lambda=-1$ as in \Cref{eq:augmented-policy-example}, the origin becomes a saddle point of $J_2$. The LQG cost decreases along $\zeta=\eta$ and increases along $\zeta=-\eta$. All policies in the plotted region $|\eta|,|\zeta|\leq0.1$ are internally stabilizing. The right panel further shows that the smallest eigenvalue of $D^2J_2(\mK_1\oplus\lambda)$ remains strictly negative throughout $\lambda\in[-10,-0.01]$. Thus, the additional controller state exposes negative curvature at a policy that is globally optimal within the first-order policy class.

Next, consider the suboptimal first-order local minimum $\mK_{1,+}$ in \Cref{example:LQG-2}.
With $b=1$, its parameters are
$
(a,c)\approx(-0.30,\,0.22), \, J_1(\mK_{1,+})\approx17.97.
$ 
As shown in the left panel of
\Cref{fig:example-2-augmentation}, appending the stable pole $\lambda=-1$ as in \Cref{eq:augmented-policy-example}  again creates a saddle point. In this case, the LQG cost decreases along $\zeta=-\eta$ and increases along $\zeta=\eta$, while all policies in the plotted region remain internally stabilizing. The right panel shows that the smallest eigenvalue of $D^2J_2(\mK_{1,+}\oplus\lambda)$ is strictly negative throughout the same range of stable poles. Hence, the spurious local minimum within $\mathcal{C}_1$ becomes
a strict saddle in $\mathcal{C}_2$ after introducing one additional controller state. \hfill $\square$
\end{example}

\begin{figure}[t]
\setlength{\abovecaptionskip}{0pt}
    \centering
    \includegraphics[width=0.36\textwidth]{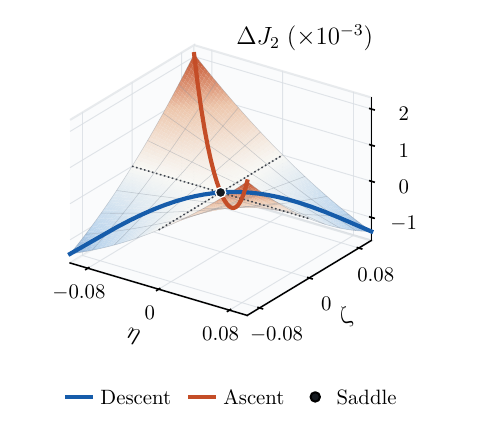}
    \hspace{10mm}
    \includegraphics[width=0.36\textwidth]{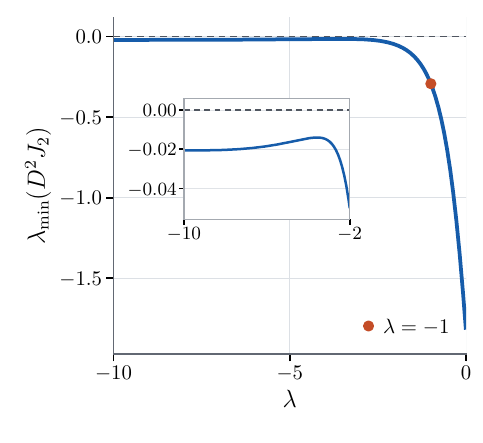}
    \caption{Augmentation of the first-order policy $\mK_1$ in \cref{example:LQG-1}. 
    Left: the LQG cost difference $\Delta J_2(\eta,\zeta)$ for
    $\lambda=-1$. The blue and orange curves show descent and ascent directions, respectively. 
    Right: the smallest eigenvalue of the full Hessian as the appended stable pole varies.}
    \label{fig:example-1-augmentation}
\end{figure}

\begin{figure}[t]
\setlength{\abovecaptionskip}{0pt}
    \centering
    \includegraphics[width=0.36\textwidth]{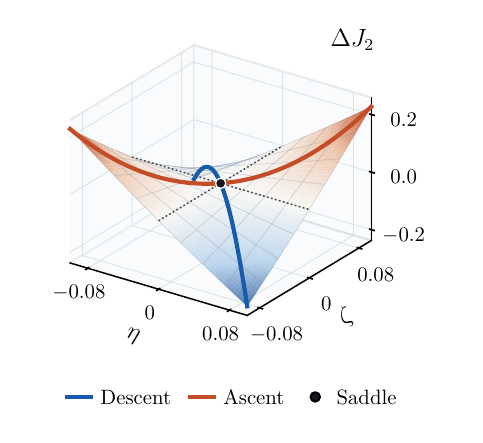}
    \hspace{10mm}
    \includegraphics[width=0.36\textwidth]{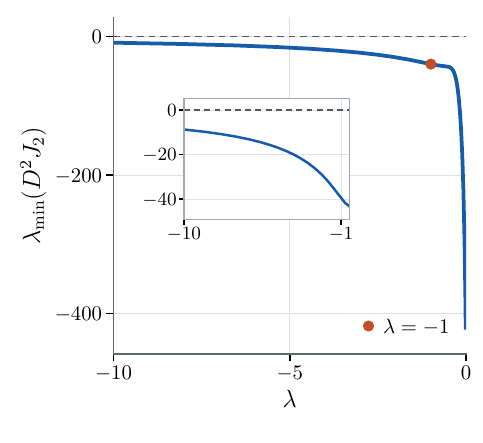}
    \caption{Augmentation of the suboptimal first-order local minimum $\mK_{1,+}$ in \cref{example:LQG-2}. Left: the LQG cost difference $\Delta J_2(\eta,\zeta)$ for
    $\lambda=-1$. The blue and orange curves show descent and ascent directions,
    respectively. Right: the smallest eigenvalue of the full Hessian as the
    appended stable pole varies.}
    \label{fig:example-2-augmentation}
\end{figure}

\section{Technical Proofs of the Main Results} \label{section:proofs}

In this section, we present the technical proofs for \cref{theorem:nonminimal-local-minimum-global,theorem:hessian-negative-direction}. Our key strategy is first to convert any local minima into their Kalman canonical form (\Cref{section:local-minima-Canonical-form}), then to characterize the Hessian of local minima in the augmented form (\Cref{section:Hessian}), and finally to connect the Hessian with the global optimality characterization in the frequency domain (\Cref{section:global-optimality}). 

\subsection{Local minima in the Kalman canonical form} \label{section:local-minima-Canonical-form}

We first note the following lemma regarding similarity transformations of local minima.
\begin{lemma}
\label{lemma:local_minimum_similarity_trans}
Let $q\in\mathbb{Z}_{>0}$ satisfy $\mathcal{C}_q\neq\varnothing$, 
and suppose $\mK=\begin{bmatrix}
0 & C_{\mK} \\ B_{\mK} & A_{\mK}
\end{bmatrix}\in\mathcal{C}_q$ is a local minimum of $J_q$. Then for any $T\in\mathrm{GL}_q$, the policy
\[
\begin{bmatrix}
0 & C_{\mK}T^{-1} \\ TB_{\mK} & TA_{\mK}T^{-1}
\end{bmatrix}
\]
is a local minimum of $J_q$.
\end{lemma}
\begin{proof}
Pick an arbitrary $T\in\mathrm{GL}_q$, and denote
\[
\mathscr{S}_T(\mK')\coloneq
\begin{bmatrix}
0 & C_{\mK'}T^{-1} \\ TB_{\mK'} & TA_{\mK'}T^{-1}
\end{bmatrix}
\qquad \text{for each}\ \ 
\mK'=\begin{bmatrix}
0 & C_{\mK'} \\ B_{\mK'} & A_{\mK'}
\end{bmatrix}\in\mathcal{C}_q.
\]
Let $\mathcal{U}$ be an open neighborhood of $\mK$ in $\mathcal{C}_q$ such that $J_q(\mK)\leq J_q(\mK')$ for all $\mK'\in\mathcal{U}$, 
and let $\mathscr{S}_T(\mathcal{U})\coloneq\setc*{\mathscr{S}_T(\mK')}{\mK'\in\mathcal{U}}$.
Since the mapping $\mathscr{S}_T$ is a bijection on $\mathcal{C}_q$, and both $\mathscr{S}_T$ and its inverse are $C^\infty$, we see that $\mathscr{S}_T(\mathcal{U})$ is an open neighborhood of $\mathscr{S}_T(\mK)$ in $\mathcal{C}_q$. Now, for any $\mK'\in \mathcal{U}$, we have
\[
J_q(\mathscr{S}_T(\mK')) = J_q(\mK')\geq J_q(\mK)=J_q(\mathscr{S}_T(\mK)),
\]
showing that $\mathscr{S}_T(\mK)$ minimizes $J_q$ over $\mathscr{S}_T(\mathcal{U})$. The proof is now complete.
\end{proof}

\Cref{lemma:local_minimum_similarity_trans} indicates that any local minimum of $J_q$ can be used to construct a family of local minima, parameterized by a similarity transformation matrix $T \in \mathrm{GL}_q$. Meanwhile, any policy $\mK \in \mathcal{C}_q$ can be converted into the so-called Kalman canonical form by some similarity transformation \cite[Theorem~3.10]{zhou1996robust}.
We define the Kalman canonical form below.

\begin{definition}
\label{definition:Kalman_canonical_form}
Let $q\in\mathbb{Z}_{>0}$ satisfy $\mathcal{C}_q\neq\varnothing$. We say that $\mK=\begin{bmatrix}
0 & C_{\mK} \\ B_{\mK} & A_{\mK}
\end{bmatrix}\in\mathcal{C}_q$ is in the Kalman canonical form, if the matrices $A_{\mK},B_{\mK},C_{\mK}$ have the following structural forms  
\begin{equation}
\label{eq:Kalman_canonical_form_controller}
A_{\mK} \!=\! \begin{bmatrix}
A_{\mK,0} & 0 & A_{\mK,02} & 0 \\
A_{\mK,10} & A_{\mK,1} & A_{\mK,12} & A_{\mK,13} \\
0 & 0 & A_{\mK,2} & 0 \\
0 & 0 & A_{\mK,32} & A_{\mK,3}
\end{bmatrix},
B_{\mK} \!=\! \begin{bmatrix}
B_{\mK,0} \\ B_{\mK,1} \\ 0_{p_{\mK,2}\times \ell} \\ 0_{p_{\mK,3}\times \ell}
\end{bmatrix},
C_{\mK} \!=\! \begin{bmatrix}
C_{\mK,0} & 0_{m\times p_{\mK,1}} & C_{\mK,2} & 0_{m\times p_{\mK,3}}
\end{bmatrix}.
\end{equation}
Here, $(A_{\mK,0},B_{\mK,0})$ is controllable and $(C_{\mK,0},A_{\mK,0})$ is observable; $p_{\mK,i}$ denotes the number of rows of $A_{\mK,i}$ for $i=1,2,3$; $B_{\mK,1}\in\mathbb{R}^{p_{\mK,1}\times\ell}$, and $C_{\mK,2}\in\mathbb{R}^{m\times p_{\mK,2}}$. 
\end{definition}

In \Cref{eq:Kalman_canonical_form_controller}, the dimension of the block $A_{\mK,0}$ is denoted by $r_\mK$, which also equals the number of rows of $B_{\mK,0}$ and the number of columns of $C_{\mK,0}$. The value $r_\mK$ is also the dimension of the minimal realization of $\mK$. By definition, we always have $q = r_{\mK} + p_{\mK,1} + p_{\mK,2} +p_{\mK,3}$. From a system-theoretic viewpoint, the Kalman canonical form \Cref{eq:Kalman_canonical_form_controller} exactly isolates controllable and observable states, controllable but unobservable states,  uncontrollable but observable states, and uncontrollable and unobservable states.

By \cite[Theorem~3.10]{zhou1996robust}, any $\mK\in\mathcal{C}_q$ can be converted into the Kalman canonical form by some similarity transformation. \Cref{lemma:local_minimum_similarity_trans} then indicates that, without loss of generality, we only need to consider local minima in the Kalman canonical form.
For notational convenience, when $\mK$ is in the Kalman canonical form~\cref{eq:Kalman_canonical_form_controller}, we define $\mathcal{I}_\mK\coloneq\setc*{i}{p_{\mK,i}>0}$, and introduce the linear space
\[
\mathcal{W}_\mK \coloneq
\setc*{(\Theta_i)_{i\in\mathcal{I}_\mK}}{\Theta_i\in \mathbb{R}^{p_{\mK,i}\times p_{\mK,i}}\text{ for each }i\in\mathcal{I}_\mK}.
\]
We will use $\mathcal{W}_\mK$ to add perturbations to the uncontrollable or unobservable blocks $A_{\mK,i}$ for $i\in\mathcal{I}_\mK$.  
In addition, we let $\mathfrak{m}(\mK)$ denote the controllable and observable policy\footnote{
We allow the case where $A_{\mK,0}$ is not present in~\cref{eq:Kalman_canonical_form_controller}. In this case, $r_\mK=0$ and $\mathfrak{m}(\mK)=0_{m\times\ell}$.
}
\[
\mathfrak{m}(\mK) \coloneq \begin{bmatrix}
0 & C_{\mK,0} \\ B_{\mK,0} & A_{\mK,0}
\end{bmatrix}.
\]

\begin{lemma}
\label{lemma:Kalman_canonical_diagonal_Hurwitz}
Given $q\in\mathbb{Z}_{>0}$, suppose $\mK\in\mathcal{C}_q$ is in the Kalman canonical form~\cref{eq:Kalman_canonical_form_controller}. Then $A_{\mK,i}$ is Hurwitz stable for each $i\in\mathcal{I}_\mK$, and
\[
\begin{bmatrix}
A & BC_{\mK,0} \\ B_{\mK,0}C & A_{\mK,0}
\end{bmatrix}
\]
is Hurwitz stable.
\end{lemma}
\begin{proof}
Straightforward calculation establishes the identity
\begin{equation}
\label{eq:closed_loop_A}
A_{\mathrm{cl},\mK}
=\begin{bmatrix}
A & BC_\mK \\
B_\mK C & A_\mK
\end{bmatrix}
=\begin{bmatrix}
A & BC_{\mK,0} & 0 & BC_{\mK,2} & 0 \\
B_{\mK,0}C & A_{\mK,0} & 0 & A_{\mK,02} & 0 \\
B_{\mK,1}C & A_{\mK,10} & A_{\mK,1} & A_{\mK,12} & A_{\mK,13} \\
0 & 0 & 0 & A_{\mK,2} & 0 \\
0 & 0 & 0 & A_{\mK,32} & A_{\mK,3}
\end{bmatrix}.
\end{equation}
We first note that $A_{\mathrm{cl},\mK}$ can be partitioned as $A_{\mathrm{cl},\mK}
=\begin{bmatrix}
\bar{A}_{11} & \ast \\
0 & \bar{A}_{22}
\end{bmatrix}$, where
\[
\bar{A}_{11} = \begin{bmatrix}
A & BC_{\mK,0} & 0 \\
B_{\mK,0}C & A_{\mK,0} & 0 \\
B_{\mK,1}C & A_{\mK,10} & A_{\mK,1}
\end{bmatrix},
\qquad
\bar{A}_{22} = \begin{bmatrix}
A_{\mK,2} & 0 \\
A_{\mK,32} & A_{\mK,3}
\end{bmatrix}.
\]
Therefore
$
\operatorname{eig}(A_{\mathrm{cl},\mK})
=\operatorname{eig}(\bar{A}_{11})\cup\operatorname{eig}(\bar{A}_{22}).
$ 
Now $\bar{A}_{22}$ is already in the block triangular form, indicating that $\operatorname{eig}(\bar{A}_{22})=\operatorname{eig}(A_{\mK,2})\cup \operatorname{eig}(A_{\mK,3})$. For $\bar{A}_{11}$, we may further partition it as
\[
\bar{A}_{11}=\begin{bmatrix}
A_{\mathrm{cl},\mathfrak{m}(\mK)} & 0 \\
\ast & A_{\mK,1}
\end{bmatrix},
\qquad\text{where}\quad
A_{\mathrm{cl},\mathfrak{m}(\mK)} = 
\begin{bmatrix}
A & BC_{\mK,0} \\ B_{\mK,0}C & A_{\mK,0}
\end{bmatrix}.
\]
Thus $\operatorname{eig}(\bar{A}_{11})=\operatorname{eig}(A_{\mathrm{cl},\mathfrak{m}(\mK)})\cup\operatorname{eig}(A_{\mK,1})$. Summarizing the above results, we obtain
\begin{equation}
\label{eq:closed_loop_eigenvalues}
\operatorname{eig}(A_{\mathrm{cl},\mK})
=\operatorname{eig}\!\left(\begin{bmatrix}
A & BC_{\mK,0} \\ B_{\mK,0}C & A_{\mK,0}
\end{bmatrix}\right)\cup \bigcup_{i\in\mathcal{I}_\mK}\operatorname{eig}(A_{\mK,i}).
\end{equation}

Finally, $\mK\in\mathcal{C}_q$ implies that $\operatorname{eig}(A_{\mathrm{cl},\mK})$ is a subset of the open left half-plane of $\mathbb{C}$. Together with \cref{eq:closed_loop_eigenvalues}, we complete the proof.
\end{proof}

Given $\mK=\begin{bmatrix}
0 & C_{\mK} \\ B_{\mK} & A_{\mK}
\end{bmatrix}\in\mathcal{C}_q$ in the Kalman canonical form~\cref{eq:Kalman_canonical_form_controller}, for all $\epsilon\in[0,1]$ and $\Theta=(\Theta_i)_{i\in\mathcal{I}_\mK}\in\mathcal{W}_\mK$, we shall denote a family of policies as 
\begin{equation} \label{eq:perturbed-policies}
\mK(\Theta,\epsilon)\coloneq
\begin{bmatrix}
0 & C_{\mK}(\epsilon) \\ B_{\mK}(\epsilon) & A_{\mK}(\Theta,\epsilon)
\end{bmatrix},
\end{equation}
where
\[
A_{\mK}(\Theta,\epsilon) \coloneq
\begin{bmatrix}
A_{\mK,0} & 0 & \epsilon A_{\mK,02} & 0 \\
\epsilon A_{\mK,10} & A_{\mK,1}+\Theta_1 & \epsilon^2 A_{\mK,12} & \epsilon A_{\mK,13} \\
0 & 0 & A_{\mK,2}+\Theta_2 & 0 \\
0 & 0 & \epsilon A_{\mK,32} & A_{\mK,3}+\Theta_3
\end{bmatrix}, \qquad B_{\mK}(\epsilon) \coloneq \begin{bmatrix}
B_{\mK,0} \\ \epsilon B_{\mK,1} \\ 0_{p_{\mK,2}\times\ell} \\ 0_{p_{\mK,3}\times\ell}
\end{bmatrix},
\]
and
$
C_{\mK}(\epsilon) \coloneq \begin{bmatrix}
C_{\mK,0} & 0_{m\times p_{\mK,1}} & \epsilon C_{\mK,2} & 0_{m\times p_{\mK,3}}
\end{bmatrix}.
$ 
Note that the policy $\mK(\Theta,0)$ can be written as
\begin{equation} \label{eq:fully-decoupled-policies}
\mK(\Theta,0)=\begin{bmatrix}
0 & C_{\mK,0} & 0 \\
B_{\mK,0} & A_{\mK,0} & 0 \\
0 & 0 & \Lambda
\end{bmatrix} = \mathfrak{m}(\mK)\oplus \Lambda
\end{equation}
for some square matrix $\Lambda$. This form \cref{eq:fully-decoupled-policies} has only two blocks: controllable and observable states, and uncontrollable and unobservable states. We simply denote $\mK(\Theta)\coloneq\mK(\Theta,1)$ and $A_\mK(\Theta)\coloneq A_\mK(\Theta,1)$.  

This family of policies \cref{eq:perturbed-policies} has the same LQG cost whenever $\mK(\Theta,\epsilon)\in\mathcal{C}_q$. 

\begin{lemma} Given $q\in\mathbb{Z}_{>0}$, suppose $\mK\in\mathcal{C}_q$ is in the Kalman canonical form~\cref{eq:Kalman_canonical_form_controller}. Define $\mK(\Theta,\epsilon)$ as in \Cref{eq:perturbed-policies}. Then, we have 
    \[
J_q(\mK(\Theta,\epsilon))=J_q(\mK)
\]
for all $\epsilon\in[0,1]$ and all $\Theta\in\mathcal{W}_\mK$ such that $\mK(\Theta,\epsilon)\in\mathcal{C}_q$.
\end{lemma}
\begin{proof}
The transfer matrix of the dynamic controller parameterized by $\mK(\Theta,\epsilon)$ is given by $C_{\mK,0}(sI-A_{\mK,0})^{-1}B_{\mK,0}$, which is independent of $\Theta\in\mathcal{W}_\mK$ and $\epsilon\in[0,1]$. Thus $\mK(\Theta,\epsilon)$ induces the same LQG cost as $\mK(0,1)=\mK$.
\end{proof}

For any local minimum $\mK\in\mathcal{C}_q$, its Hessian must be positive semidefinite. If it is further an uncontrollable or unobservable policy in the Kalman canonical form, then we can ensure that there exists a class of perturbed policies $\mK(\Theta,0)$ with positive semidefinite Hessians.

\begin{lemma}
\label{lemma:hessian_neighborhood_Kalman}
Let $q\in\mathbb{Z}_{>0}$ satisfy $\mathcal{C}_q\neq\varnothing$,
and suppose $\mK=\begin{bmatrix}
0 & C_{\mK} \\ B_{\mK} & A_{\mK}
\end{bmatrix}\in\mathcal{C}_q$ is a local minimum of $J_q$ in the Kalman canonical form~\cref{eq:Kalman_canonical_form_controller} which is uncontrollable or unobservable. 
Then, there exists an open neighborhood $\mathcal{O}_\mK$ of $0$ in $\mathcal{W}_\mK$ such that
\[
\mK(\Theta,0)\in\mathcal{C}_q
\qquad\text{and}\qquad
D^2 J_q(\mK(\Theta,0))[\Delta,\Delta]\geq 0,
\quad\forall \Delta\in\mathcal{V}_q
\]
for all $\Theta\in\mathcal{O}_\mK$.
\end{lemma}
\begin{proof}
Note that $\mathcal{I}_\mK\neq\varnothing$ and hence $\mathcal{W}_\mK\neq\varnothing$ because $\mK$ is not controllable or observable. Let $\mathcal{U}$ be an open neighborhood of $\mK$ in $\mathcal{C}_q$ such that $J_q(\mK)\leq J_q(\mK')$ for all $\mK'\in\mathcal{U}$. Since the mapping $\Theta\mapsto\mK(\Theta)$ is continuous, we can find $\rho>0$ and subsequently define
\[
\mathcal{O}_\mK = \setc*{(\Theta_i)_{i\in\mathcal{I}_\mK}\in\mathcal{W}_\mK}{\|\Theta_i\|_F<\rho\text{ for all }i\in\mathcal{I}_\mK},
\]
such that $\setc*{\mK(\Theta)}{\Theta\in\mathcal{O}_\mK}$ is a subset of $\mathcal{U}$. Then for each $\Theta\in\mathcal{O}_\mK$, we have $\mK(\Theta)\in\mathcal{U}$ and
\[
J_q(\mK(\Theta))=J_q(\mK)\leq J_q(\mK'),\qquad\forall\mK'\in\mathcal{U},
\]
showing that $\mK(\Theta)$ is a local minimum of $J_q$ for all $\Theta\in\mathcal{O}_\mK$. 

Next, we show that $\mK(\Theta,\epsilon)$ is a local minimum of $J_q$ for any $\epsilon\in(0,1]$ and $\Theta\in\mathcal{O}_\mK$. Let
\[
T_\epsilon = \operatorname{diag}(
I_{r_{\mK}},\epsilon I_{p_{\mK,1}},\epsilon^{-1} I_{p_{\mK,2}},I_{p_{\mK,3}})
\]
for $\epsilon\in(0,1]$, and it is easy to check that
\[
A_\mK(\Theta,\epsilon) = T_\epsilon A_{\mK}(\Theta)T_\epsilon^{-1},
\quad
B_\mK(\epsilon)=T_\epsilon B_\mK,
\quad
C_\mK(\epsilon) = C_\mK T_\epsilon^{-1}.
\]
In other words, for each $\epsilon\in(0,1]$,  $\mK(\Theta,\epsilon)$ and $\mK(\Theta)$ are related by the similarity transformation associated with $T_\epsilon$. 
By \cref{lemma:local_minimum_similarity_trans}, we see that $\mK(\Theta,\epsilon)$ is a local minimum of $J_q$ for all $\Theta\in\mathcal{O}_\mK$ and $\epsilon\in(0,1]$, and consequently the Hessian $D^2 J_q(\mK(\Theta,\epsilon))$ is positive semidefinite when $\Theta\in\mathcal{O}_\mK$ and $\epsilon\in(0,1]$.

We then show that $\mK(\Theta,0)\in\mathcal{C}_q$. Noting that $\mK(\Theta)$ is in the Kalman canonical form and that $\mK(\Theta)\in\mathcal{C}_q$ for each $\Theta\in\mathcal{O}_\mK$, by using \cref{lemma:Kalman_canonical_diagonal_Hurwitz}, we see that $A_{\mK,i}+\Theta_i$ for each $i\in\mathcal{I}_\mK$ and $\begin{bmatrix}
A & BC_{\mK,0} \\ B_{\mK,0}C & A_{\mK,0}
\end{bmatrix}$ are Hurwitz stable whenever $\Theta\in\mathcal{O}_\mK$. Meanwhile, by definition \cref{eq:perturbed-policies} (see also \cref{eq:fully-decoupled-policies}), we have  
\begin{equation*}
A_{\mathrm{cl},\mK(\Theta,0)}=
\begin{bmatrix}
A & BC_\mK(0) \\
B_\mK(0) C & A_\mK(\Theta,0)
\end{bmatrix}
= \operatorname{diag}\left( \begin{bmatrix}
 A & BC_{\mK,0} \\
B_{\mK,0}C & A_{\mK,0}\end{bmatrix}, A_{\mK,1}+\Theta_1, A_{\mK,2}+\Theta_2, A_{\mK,3}+\Theta_3\right).
\end{equation*}
Thus $A_{\mathrm{cl},\mK(\Theta,0)}$ is Hurwitz stable, and we can conclude that $\mK(\Theta,0)\in\mathcal{C}_q$ for every $\Theta\in\mathcal{O}_\mK$.

Now let $\Delta\in\mathcal{V}_q$ and $\Theta\in\mathcal{O}_\mK$ be arbitrary and fixed. For each $\epsilon\in[0,1]$, define
\[
h(\epsilon) = D^2 J_q(\mK(\Theta,\epsilon))[\Delta,\Delta].
\]
We then have $h(\epsilon)\geq 0$ for all $\epsilon\in(0,1]$. Since $J_q$ is real-analytic over $\mathcal{C}_q$, the function $h$ is continuous on $[0,1]$. Thus
\[
h(0)=\lim_{\epsilon\downarrow 0}h(\epsilon)\geq 0,
\]
which completes the proof.
\end{proof}

Note that when $\mK\in\mathcal{C}_q$ is in the Kalman canonical form, the policy $\mK(\Theta,0)$ is in the augmented form \cref{eq:fully-decoupled-policies}.
It turns out that the Hessian of policies in the augmented form has further structures, and the next section will reveal such structures.

\subsection{Hessian for local minima in the augmented form} \label{section:Hessian}

For any $\mK
=
\begin{bmatrix}
0 & C_{\mK} \\
B_{\mK} & A_{\mK}
\end{bmatrix}
\in\mathcal{C}_q$ and $\Lambda\in\mathbb{R}^{p\times p}$, recall that the augmented policy $\mK\oplus\Lambda$ is defined in \cref{eq:augmented-policy}.
Given $\mK\in\mathcal{C}_q$ and $\Lambda\in\mathbb{R}^{p\times p}$, we have $\mK\oplus\Lambda\in\mathcal{C}_{q+p}$ if and only if $\Lambda\in\mathfrak{H}_p$ (i.e., $\Lambda$ is Hurwitz stable).

Our main technical result in this subsection is the following Hessian characterization. 

\begin{proposition}
\label{lemma:vanishing_hessian_augmented_scalar}
Let $q\in\mathbb{Z}_{>0}$ satisfy $\mathcal{C}_q\neq\varnothing$,
and suppose $\mK\in\mathcal{C}_{q}$ is a local minimum of $J_q$ in the Kalman canonical form~\cref{eq:Kalman_canonical_form_controller} which is uncontrollable or unobservable. Recall that $\mathfrak{m}(\mK)$ denotes the controllable and observable policy
\[
\mathfrak{m}(\mK) \coloneq \begin{bmatrix}
0 & C_{\mK,0} \\ B_{\mK,0} & A_{\mK,0}
\end{bmatrix},
\]
and that $r_\mK$ denotes the dimension of $A_{\mK,0}$. Then for any $\lambda\in(-\infty,0)$, we have
\[
D^2 J_q(\mathfrak{m}(\mK)\oplus\lambda I_{q-r_\mK})[\Delta,\Delta]=0, \qquad \forall \, \Delta = \begin{bmatrix}
0_{(m+r_{\mK})\times(\ell+r_{\mK})} & \Delta_{12} \\
\Delta_{21} & 0_{(q - r_{\mK})\times(q - r_{\mK})}
\end{bmatrix} 
\]
with $\Delta_{12} \in \mathbb{R}^{(m+r_{\mK}) \times (q - r_{\mK})}$, and  $\Delta_{21} \in \mathbb{R}^{(q - r_{\mK}) \times (\ell+r_{\mK})}$.
\end{proposition}

The proof of this result requires \Cref{lemma:hessian_neighborhood_Kalman}, some further Hessian characterizations below, as well as the zero set of real analytic functions. Both $\mathfrak{m}(\mK)\oplus\lambda I_{q-r_\mK}$ in \Cref{lemma:vanishing_hessian_augmented_scalar} and $\mK(\Theta,0)$ in \Cref{lemma:hessian_neighborhood_Kalman} (also see \cref{eq:fully-decoupled-policies}) are policies in the augmented form, both of which contain the same minimal realization $\mathfrak{m}(\mK)$. One good feature in \Cref{lemma:vanishing_hessian_augmented_scalar} is that the policy $\mathfrak{m}(\mK)\oplus\lambda I_{q-r_\mK}$ allows using any diagonal block $\lambda I_{q-r_\mK}$ with $\lambda\in(-\infty,0)$, and that its second-order directional derivatives 
must be exactly zero for all restricted directions $\Delta$  with zero diagonal blocks.
This refinement will be one key step in establishing \cref{theorem:nonminimal-local-minimum-global,theorem:hessian-negative-direction}.

In the rest of this subsection, we present further Hessian characterizations to prove  \Cref{lemma:vanishing_hessian_augmented_scalar}. 
Given any $\mK
\in\mathcal{C}_q$ and $\Lambda\in\mathfrak{H}_p$, we focus on $D^2 J_{q+p}(\mK\oplus\Lambda)(\Delta,\Delta)$ in which $\Delta$ has the~form:
\begin{equation}
\label{eq:hessian_Delta_partition}
\Delta = \begin{bmatrix}
0_{(m+q)\times(\ell+q)} & \Delta_{12} \\
\Delta_{21} & 0_{p\times p}
\end{bmatrix}.
\end{equation}
Here $\Delta_{12}\in\mathbb{R}^{(m+q)\times p}$ and $\Delta_{21}\in\mathbb{R}^{p\times(\ell+q)}$.
The linear space consisting of all $\Delta$ of the above form will be denoted by $\mathcal{M}_{q,p}$.

\begin{lemma}
\label{lemma:hessian_augmented_general}
Let $q\in\mathbb{Z}_{\geq 0}$ satisfy $\mathcal{C}_q\neq\varnothing$, and let $p\in\mathbb{Z}_{>0}$ be arbitrary. Given $\mK\in\mathcal{C}_q$ and $\Lambda\in\mathfrak{H}_p$, there exists a bilinear form $\mathscr{B}_{\mK,\Lambda}:\mathbb{R}^{(m+q)\times p}\times\mathbb{R}^{p\times(\ell+q)}\rightarrow\mathbb{R}$ such that
\[
D^2 J_{q+p}(\mK\oplus\Lambda)[\Delta,\Delta]=\mathscr{B}_{\mK,\Lambda}\!\left[\Delta_{12},\Delta_{21}\right]
\]
for all
\[
\Delta = \begin{bmatrix}
0_{(m+q)\times(\ell+q)} & \Delta_{12} \\
\Delta_{21} & 0_{p\times p}
\end{bmatrix}\in\mathcal{M}_{q,p}.
\]
\end{lemma}
\begin{proof}
Denote $\mathscr{H}=D^2 J_{q+p}(\mK\oplus\Lambda)$ temporarily for simplicity. 
Define the projection operators $\mathscr{P}_{12},\mathscr{P}_{21}:\mathcal{M}_{q,p}\rightarrow\mathcal{M}_{q,p}$ by
\[
\mathscr{P}_{12}\Delta = \begin{bmatrix}
0 & \Delta_{12} \\
0 & 0
\end{bmatrix},
\quad
\mathscr{P}_{21}\Delta = \begin{bmatrix}
0 & 0 \\
\Delta_{21} & 0
\end{bmatrix},
\quad\forall \Delta = \begin{bmatrix}
0 & \Delta_{12} \\
\Delta_{21} & 0
\end{bmatrix}\in\mathcal{M}_{q,p}.
\]
Also, for each $\Delta\in\mathcal{M}_{q,p}$, denote
\[
\widehat{J}_\Delta(t) = J_{q+p}(\mK\oplus\Lambda+t\Delta)
\]
for $t\in\mathbb{R}$ in a neighborhood of $0$ such that $\mK\oplus\Lambda+t\Delta\in\mathcal{C}_{q+p}$. We then have $\widehat{J}''_\Delta(0)=\mathscr{H}[\Delta,\Delta]$.

We first note that, if $\mathscr{P}_{21}\Delta=0$ (i.e., the $\Delta_{21}$ block of $\Delta$ in~\cref{eq:hessian_Delta_partition} is zero), then
\[
\mK\oplus\Lambda+\Delta = \begin{bmatrix}
0 & C_{\mK} & \ast \\
B_{\mK} & A_{\mK} & \ast \\
0 & 0 & \Lambda
\end{bmatrix},
\]
where we use $\ast$ to represent possibly nonzero blocks whose specific values are not relevant. Consequently, the transfer matrix from $y$ to $u$ corresponding to the policy $\mK\oplus\Lambda+\Delta$ is given by\footnote{For the policy $\mK=0_{m\times\ell}$, this transfer matrix is simply $0_{m\times\ell}$.}
\[
\begin{aligned}
 C_{\mK\oplus\Lambda+\Delta}(sI_{q+p}-A_{\mK\oplus\Lambda+\Delta})^{-1}B_{\mK\oplus\Lambda+\Delta} 
={} & \begin{bmatrix}
C_{\mK} & \ast
\end{bmatrix}
\left(\begin{bmatrix}
sI_q-A_{\mK} & \ast \\ 0 & sI_p-\Lambda
\end{bmatrix}\right)^{-1}\begin{bmatrix}
B_{\mK} \\ 0
\end{bmatrix} \\
={} & \begin{bmatrix}
C_{\mK} & \ast
\end{bmatrix}
\begin{bmatrix}
(sI_q-A_{\mK})^{-1} & \ast \\ 0 & (sI_p-\Lambda)^{-1}
\end{bmatrix}\begin{bmatrix}
B_{\mK} \\ 0
\end{bmatrix} \\
={} & C_{\mK}(sI_q-A_{\mK})^{-1}B_{\mK},
\end{aligned}
\]
which is independent of $\Delta_{12}$. Thus, when $\mathscr{P}_{21}\Delta=0$, adding the perturbation $\Delta$ to the policy $\mK\oplus\Lambda$ does not change the corresponding objective value. As a result, for every $\Delta\in\mathcal{M}_{q,p}$, we have
\[
\widehat{J}_{\mathscr{P}_{12}\Delta}(t)
=J_{q+p}(\mK\oplus\Lambda)
=J_q(\mK),
\]
which further leads to
$
\mathscr{H}[\mathscr{P}_{12}\Delta,\mathscr{P}_{12}\Delta]=
\widehat{J}_{\mathscr{P}_{12}\Delta}''(0)
=0
$ 
for all $\Delta\in\mathcal{M}_{q,p}$. 

Similarly, we can show that when $\mathscr{P}_{12}\Delta=0$, the transfer matrix from $y$ to $u$ corresponding to the policy $\mK\oplus\Lambda+\Delta$ is still $C_{\mK}(sI_q-A_{\mK})^{-1}B_{\mK}$, independent of $\Delta_{21}$. As a result, we will have
$
\mathscr{H}[\mathscr{P}_{21}\Delta,\mathscr{P}_{21}\Delta]=0 
$ 
for every $\Delta\in\mathcal{M}_{q,p}$.

Summarizing the previous results, we have
\[
\begin{aligned}
\mathscr{H}[\Delta,\Delta]
={} & \mathscr{H}[(\mathscr{P}_{12}+\mathscr{P}_{21})\Delta,(\mathscr{P}_{12}+\mathscr{P}_{21})\Delta] \\
={} & \mathscr{H}[\mathscr{P}_{12}\Delta,\mathscr{P}_{12}\Delta]+\mathscr{H}[\mathscr{P}_{21}\Delta,\mathscr{P}_{21}\Delta]+2\mathscr{H}[\mathscr{P}_{12}\Delta,\mathscr{P}_{21}\Delta] \\
={} & 2\mathscr{H}[\mathscr{P}_{12}\Delta,\mathscr{P}_{21}\Delta].
\end{aligned}
\]
Letting
\[
\mathscr{B}_{\mK,\Lambda}[\Delta_{12},\Delta_{21}]=2\mathscr{H}\!\left[
\begin{bmatrix}
0 & \Delta_{12} \\
0 & 0
\end{bmatrix},
\begin{bmatrix}
0 & 0 \\
\Delta_{21} & 0
\end{bmatrix}
\right]
\]
completes the proof.
\end{proof}
\begin{corollary}
\label{corollary:vanishing_hessian_augmented}
Let $q\in\mathbb{Z}_{\geq 0}$ satisfy $\mathcal{C}_q\neq\varnothing$, and let $p\in\mathbb{Z}_{>0}$ be arbitrary. Given $\mK\in\mathcal{C}_q$ and $\Lambda\in\mathfrak{H}_p$,
we have $
D^2 J_{q+p}(\mK\oplus\Lambda)[\Delta,\Delta]= 0
$ 
for all $\Delta\in\mathcal{M}_{q,p}$ if one of the following holds:
\begin{enumerate}
\item $
D^2 J_{q+p}(\mK\oplus\Lambda)[\Delta,\Delta]\geq 0
$
for all $\Delta\in\mathcal{M}_{q,p}$.
\item $
D^2 J_{q+p}(\mK\oplus\Lambda)[\Delta,\Delta]\leq 0
$
for all $\Delta\in\mathcal{M}_{q,p}$.
\end{enumerate}
\end{corollary}
\begin{proof}
We shall prove the following equivalent statement: If $D^2 J_{q+p}(\mK\oplus\Lambda)[\Delta,\Delta]\neq 0$ for some $\Delta\in\mathcal{M}_{q,p}$, then there exists some $\Delta'\in\mathcal{M}_{q,p}$ such that $D^2 J_{q+p}(\mK\oplus\Lambda)[\Delta',\Delta']\neq 0$ with a sign different from $D^2 J_{q+p}(\mK\oplus\Lambda)[\Delta,\Delta]$.

Suppose $D^2 J_{q+p}(\mK\oplus\Lambda)[\Delta,\Delta]>0$ for some $\Delta\in\mathcal{M}_{q,p}$. By \cref{lemma:hessian_augmented_general}, we have  $\mathscr{B}_{\mK,\Lambda}\!\left(\Delta_{12},\Delta_{21}\right)>0$. By letting
\[
\Delta' = \begin{bmatrix}
0 & -\Delta_{12} \\ \Delta_{21} & 0
\end{bmatrix},
\]
we then get
\[
D^2 J_{q+p}(\mK\oplus\Lambda)[\Delta',\Delta'] = \mathscr{B}_{{\mK},\Lambda}\!\left(-\Delta_{12},\Delta_{21}\right) = -\mathscr{B}_{{\mK},\Lambda}\!\left(\Delta_{12},\Delta_{21}\right)<0.
\]
The case where $D^2 J_{q+p}(\mK\oplus\Lambda)[\Delta,\Delta]<0$ for some $\Delta\in\mathcal{M}_{q,p}$ can be proved similarly.
\end{proof}

To prove \cref{lemma:vanishing_hessian_augmented_scalar}, we need a lemma on the zero set of real analytic functions, and also a lemma establishing the connectedness and openness of the set of Hurwitz stable matrices.
\begin{lemma}[\cite{mityagin2020zero}]
\label{lemma:zero_set_analytic_function}
Let $f:\mathcal{D}\rightarrow\mathbb{R}$ be a real analytic function, with $\mathcal{D}$ being connected and open in $\mathbb{R}^d$. If $f$ is not identically zero on $\mathcal{D}$, then $\setc*{x\in\mathcal{D}}{f(x)=0}$ has Lebesgue measure zero.
\end{lemma}
\begin{lemma}[\cite{duan1998note}]
For any positive integer $p$, the set $\mathfrak{H}_p$ is connected and open in $\mathbb{R}^{p\times p}$.
\end{lemma}
\begin{proof}[Proof of \cref{lemma:vanishing_hessian_augmented_scalar}]
We have $\mathcal{I}_\mK\neq\varnothing$ because $\mK$ is uncontrollable or unobservable.
Fix an arbitrary $\Delta\in\mathcal{M}_{r_\mK,q-r_\mK}$. Let
\[
\begin{aligned}
\mathcal{W}_\mK^{\mathrm{H}} ={} & \setc*{\Theta\in\mathcal{W}_\mK}{\mK(\Theta,0)\in\mathcal{C}_q} \\
={} & \setc*{\Theta\in\mathcal{W}_\mK}{
A_{\mK,i}+\Theta_i\text{ is Hurwitz stable for all }i\in\mathcal{I}_\mK
},
\end{aligned}
\]
and
\[
h(\Theta) = D^2 J_q(\mK(\Theta,0))[\Delta,\Delta],\qquad\forall\Theta\in\mathcal{W}_\mK^{\mathrm{H}}.
\]
The fact that $J_q$ is real analytic over $\mathcal{C}_q$ implies that $h$ is real analytic over $\Theta\in\mathcal{W}_\mK^{\mathrm{H}}$. In addition, the set $\mathcal{W}_\mK^{\mathrm{H}}$ is open and connected since $\mathfrak{H}_{p}$ is open and connected for any positive integer $p$. Then, noting that for any $\Theta\in\mathcal{W}_\mK^{\mathrm{H}}$, the policy $\mK(\Theta,0)$ can be written as $\mathfrak{m}(\mK)\oplus\Lambda$ for some $\Lambda\in\mathfrak{H}_{q-r_\mK}$, by \cref{lemma:hessian_neighborhood_Kalman} and \cref{corollary:vanishing_hessian_augmented}, we have
\[
h(\Theta)=0,\qquad\forall\Theta\in\mathcal{O}_\mK.
\]
Therefore, the set $\setc*{\Theta\in\mathcal{W}_\mK^{\mathrm{H}}}{h(\Theta)=0}$ has a non-empty open subset $\mathcal{O}_\mK$ and consequently has a positive Lebesgue measure. By \cref{lemma:zero_set_analytic_function}, we see that $h$ is identically zero on $\mathcal{W}^{\mathrm{H}}_\mK$. We complete the proof by noting that
\[
D^2 J_q(\mathfrak{m}(\mK)\oplus\lambda I_{q-r_\mK})[\Delta,\Delta]
=h(\Theta_\lambda)
\]
with $\Theta_\lambda=(-A_{\mK,i}+\lambda I_{p_{\mK,i}})_{i\in\mathcal{I}_\mK}\in \mathcal{W}^{\mathrm{H}}_\mK$ for all $\lambda<0$.
\end{proof}

\subsection{Proofs of \cref{theorem:nonminimal-local-minimum-global,theorem:hessian-negative-direction}}
\label{section:global-optimality}

The final technical ingredients of our proofs involve a frequency-domain characterization of the global optimality of $\mK$ and an explicit formula of the Hessian at $\mK\oplus \lambda I_p$. 
\begin{theorem}
\label{theorem:global_optimality_RK}
Let $q\in\mathbb{Z}_{\geq 0}$ satisfy $\mathcal{C}_q\neq\varnothing$.
Given $\mK=\begin{bmatrix}
0 & C_\mK \\ B_\mK & A_\mK
\end{bmatrix}\in\mathcal{C}_q$, let $X_\mK$ and $Y_\mK$ be solutions to the Lyapunov equations:
\begin{subequations} \label{eq:closed-loop_Lyapunov}
\begin{align}
A_{\mathrm{cl},\mK}X_\mK + X_\mK A_{\mathrm{cl},\mK}^{\tran}
+B_{\mathrm{cl},\mK} B_{\mathrm{cl},\mK}^{\tran} ={} & 0,  
\label{eq:closed-loop_Lyapunov_X}\\
A_{\mathrm{cl},\mK}^{\tran} Y_\mK + Y_\mK A_{\mathrm{cl},\mK} + C_{\mathrm{cl},\mK}^{\tran} C_{\mathrm{cl},\mK} ={} & 0.
\label{eq:closed-loop_Lyapunov_Y}
\end{align}
\end{subequations}
Furthermore, define the transfer matrix
\[
\mathbf{R}_\mK(s) \coloneq \left(\begin{bmatrix}
0_{m\times n} & \!\!RC_\mK \\
0_{q\times n} & \!\!0_{q\times q}
\end{bmatrix}+\begin{bmatrix}
B^\tran & \!\!0_{m\times q} \\ 0_{q\times n} & \!\!I_q
\end{bmatrix}Y_\mK\right)\!\left(sI-A_{\mathrm{cl},\mK}\right)^{-1}
\!\left(\begin{bmatrix}
0_{n\times \ell} & \!\!0_{n\times q} \\
B_\mK V & \!\!0_{q\times q}
\end{bmatrix}
+X_\mK \begin{bmatrix}
C^\tran & \!\!0_{n\times q} \\ 0_{q\times \ell} & \!\!I_q
\end{bmatrix}\right).
\]
Then $\mK$ is a globally optimal policy if and only if
$\mathbf{R}_\mK(s) \equiv 0$.
\end{theorem}

A version of \Cref{theorem:global_optimality_RK} appeared recently in \cite[Theorem 3.4]{li2025policy}, under the standing assumption that the controller order $q$ is large enough so that $\mathcal{C}_q$ contains a globally optimal LQG controller. Our \Cref{theorem:global_optimality_RK} works for any controller order $q$ as long as $\mathcal{C}_q \neq\varnothing$. 
The proof follows the same high-level strategy as \cite{li2025policy} by exploiting the convexity of the LQG problem \cref{eq:LQG-control} revealed by the Youla parameterization in the frequency domain. We provide a self-contained proof in \Cref{appendix:proof_global_optimality_RK}, which is more direct than the one in \cite{li2025policy}.\footnote{The proof in \cite{li2025policy} relies on \cite[Theorem~2]{zheng2022escaping} for the necessity direction.} In particular, given any stabilizing policy $\mK \in \mathcal{C}_q$, the Youla parameterization centered at $\mK$ represents all stabilizing policies, of arbitrary orders, through a stable Youla parameter $\mathbf{Q}$, with $\mathbf{Q} =0$ corresponding to $\mK$. Under this parameterization, the LQG problem~\cref{eq:LQG-control} becomes a convex quadratic optimization problem in the Youla parameter $\mathbf{Q}$. Consequently, $\mK$ is globally optimal if and only if the first-order derivative of the Youla-parameterized objective vanishes at $\mathbf{Q} =0$. We then show that this first-order condition is precisely equivalent to $\mathbf{R}_\mK(s) \equiv 0$. The details are technically involved, which are given in \Cref{appendix:proof_global_optimality_RK}.

The next result reveals an explicit relation between the Hessian at $\mK\oplus \lambda I_p$ and the transfer matrix $\mathbf{R}_\mK$. 

\begin{theorem}
\label{lemma:hessian_augmented_scalar}
Let $q\in\mathbb{Z}_{\geq 0}$ satisfy $\mathcal{C}_q\neq\varnothing$, and let $p\in\mathbb{Z}_{>0}$. Given $\mK\in\mathcal{C}_q$, for every $\lambda<0$ and
\[
\Delta = \begin{bmatrix}
0_{(m+q)\times(\ell+q)} & \Delta_{12} \\
\Delta_{21} & 0_{p\times p}
\end{bmatrix}\in\mathcal{M}_{q,p},
\]
we have
\begin{equation} \label{eq:Hessian_augmented_scalar}
D^2J_{q+p}\big(\mK\oplus\lambda I_{p}\big)
[\Delta,\Delta]
=4\operatorname{tr}\!\left(
\Delta_{21}
\mathbf{R}_{\mK}(-\lambda)^\tran
\Delta_{12}\right).
\end{equation}
\end{theorem}

The proof is based on a careful and direct (but a bit tedious) specialization of the general Hessian formula in \cite[Lemma 9]{tang2023analysis}. 
The key observation is that this formula simplifies substantially
because of the augmented form $\mK\oplus \lambda I_p$, the simple diagonal structure $\lambda I_p$, and the selected direction $\Delta$. 
After the simplification, the explicit Hessian connection with $\mathbf{R}_\mK$ in  \cref{eq:Hessian_augmented_scalar} looks very interesting to us. We note that related Hessian calculations appear in \cite[Theorem 5]{tang2023analysis}, \cite[Theorem 2]{zheng2022escaping}, and \cite[Corollary B.1]{li2025policy}.  For completeness, we provide the detailed derivation of \Cref{lemma:hessian_augmented_scalar} in \cref{appendix:proof_hessian_augmented_scalar}. Note that \Cref{theorem:global_optimality_RK,lemma:hessian_augmented_scalar} also precisely answer a question raised in \cite[Remark 2]{zheng2022escaping}.

Now, combining \cref{lemma:vanishing_hessian_augmented_scalar}, \cref{theorem:global_optimality_RK} and \cref{lemma:hessian_augmented_scalar}, we can establish  \cref{theorem:nonminimal-local-minimum-global} that any local minimum of $J_q$ that is uncontrollable or unobservable must be globally optimal.

\begin{proof}[Proof of \cref{theorem:nonminimal-local-minimum-global}] Let $\mK \in \mathcal{C}_q$ be a local minimum of $J_q$. As suggested earlier, we can, without loss of generality, assume that $\mK$ is in the Kalman canonical form~\cref{eq:Kalman_canonical_form_controller}. This is guaranteed by  \cref{lemma:local_minimum_similarity_trans} and \cite[Theorem~3.10]{zhou1996robust}. \Cref{lemma:vanishing_hessian_augmented_scalar} then implies that, for any $\lambda\in(-\infty,0)$, we have
\[
D^2 J_q(\mathfrak{m}(\mK)\oplus\lambda I_{q-r_\mK})[\Delta,\Delta]=0
\]
for all $\Delta\in\mathcal{M}_{r_\mK,q-r_\mK}$. We then apply \cref{lemma:hessian_augmented_scalar} to obtain
\[
\operatorname{tr}\!\left(
\Delta_{21}
\mathbf{R}_{\mathfrak{m}(\mK)}(-\lambda)^\tran
\Delta_{12}\right) = 0
\]
for all $\Delta_{12}\in\mathbb{R}^{(m+r_\mK)\times (q-r_\mK)}$, $\Delta_{21}\in\mathbb{R}^{(q-r_\mK)\times(\ell+r_\mK)}$ and $\lambda\in(-\infty,0)$. For any $1\leq i\leq m+r_\mK$ and $1\leq j\leq \ell+r_\mK$, there exist certain choices of $\Delta_{12}\in\mathbb{R}^{(m+r_\mK)\times (q-r_\mK)}$ and $\Delta_{21}\in\mathbb{R}^{(q-r_\mK)\times(\ell+r_\mK)}$ such that $\operatorname{tr}\!\left(
\Delta_{21}
\mathbf{R}_{\mathfrak{m}(\mK)}(-\lambda)^\tran
\Delta_{12}\right)$ equals the $(i,j)$th entry of $\mathbf{R}_{\mathfrak{m}(\mK)}(-\lambda)$. Therefore, we have 
\begin{equation}
\label{eq:RK_vanish_negative_reals}
\mathbf{R}_{\mathfrak{m}(\mK)}(-\lambda)=0,
\qquad\forall \lambda\in(-\infty,0).
\end{equation}

We may further write the $(i,j)$th entry of $\mathbf{R}_{\mathfrak{m}(\mK)}(s)$ as a quotient $\frac{n_{ij}(s)}{d_{ij}(s)}$, where $n_{ij}(s)$ and $d_{ij}(s)$ are both polynomials. The property~\cref{eq:RK_vanish_negative_reals} then implies that $n_{ij}(s)=0$ for $s\in(0,+\infty)$. Finally, note that the identity theorem for holomorphic functions (see, e.g., \cite[Theorem 10.18]{rudin1986real}) implies that, if a polynomial vanishes on a subset of $\mathbb{C}$ that has a limit point, it must also vanish on the whole complex plane. It is not hard to see that $(0,+\infty)$ is a non-empty set of which all members are its limit points in $\mathbb{C}$. Thus the polynomial $n_{ij}(s)$ vanishes on the whole complex plane. Summarizing these results, we obtain $\mathbf{R}_{\mathfrak{m}(\mK)}=0$. Applying \cref{theorem:global_optimality_RK} gives
\[
J_q(\mK) = J_{r_\mK}(\mathfrak{m}(\mK))
=\inf_{\mK'\in\mathcal{C}_n} J_n(\mK'),
\]
which completes the proof.
\end{proof}

We are now ready to provide the proof of \Cref{theorem:hessian-negative-direction}.

\begin{proof}[Proof of \Cref{theorem:hessian-negative-direction}]
Since $\mK$ is not a globally optimal policy, by \cref{theorem:global_optimality_RK}, the transfer matrix $\mathbf{R}_\mK$ is nonzero, i.e., at least one entry of $\mathbf{R}_\mK$ is a nonzero rational function. Suppose this nonzero entry is the $(i,j)$th entry, and write it as the quotient $\frac{n_{ij}(s)}{d_{ij}(s)}$ where $n_{ij}(s)$ and $d_{ij}(s)$ are both nonzero polynomials. Then
\[
\setc*{\lambda\in(-\infty,0)}{n_{ij}(-\lambda)=0}
\]
has Lebesgue measure zero by \cref{lemma:zero_set_analytic_function}.

Next, we denote
\[
\begin{aligned}
\mathcal{Z} ={} & \setc*{\lambda\in(-\infty,0)}{D^2 J_{q+p}(\mK\oplus\lambda I_p)[\Delta,\Delta]\geq 0\text{ for all }\Delta\in \mathcal{V}_{q+p}} \\
& \cup \setc*{\lambda\in(-\infty,0)}{D^2 J_{q+p}(\mK\oplus\lambda I_p)[\Delta,\Delta]\leq 0\text{ for all }\Delta\in \mathcal{V}_{q+p}}.
\end{aligned}
\]
By \cref{corollary:vanishing_hessian_augmented}, we have
\[
\mathcal{Z} \subseteq \setc*{\lambda\in(-\infty,0)}{D^2 J_{q+p}(\mK\oplus\lambda I_p)[\Delta,\Delta]= 0\text{ for all }\Delta\in \mathcal{M}_{q,p}}.
\]
By using \cref{lemma:hessian_augmented_scalar}, we can find a particular $\Delta\in\mathcal{M}_{q,p}$ such that
\[
D^2J_{q+p}\big(\mK\oplus\lambda I_{p}\big)
[\Delta,\Delta]
=4\frac{n_{ij}(-\lambda)}{d_{ij}(-\lambda)}
\]
for each $\lambda\in(-\infty,0)$. Consequently,
\[
\mathcal{Z}\subseteq
\setc*{\lambda\in(-\infty,0)}{n_{ij}(-\lambda)=0}.
\]
We now see that $\mathcal{Z}$ is a subset of a set of measure zero, and thus has Lebesgue measure zero. We complete the proof by noting that the set~\cref{eq:augmented_scalar_negative_direction} is the complement of $\mathcal{Z}$ in $(-\infty,0)$.
\end{proof}

\begin{remark}
     The frequency-domain characterization of the global optimality of $\mK$ in \Cref{theorem:global_optimality_RK} is a central result to establish \cref{theorem:nonminimal-local-minimum-global,theorem:hessian-negative-direction}. This frequency-domain characterization is sufficient and necessary, and does not rely on the order of policies. The proof exploits hidden convexity under the Youla parameterization, which allows local first-order information to certify global optimality. This frequency-domain viewpoint  complements the recent extended convex lifting framework in \cite{zheng2026benignI,zheng2026benignII,zheng2025extended}, which exploits hidden convexity in the state-space representation. The global-optimality results for LQG control in these works rely on a nondegeneracy condition for full-order dynamic policies, which becomes irrelevant in \Cref{theorem:global_optimality_RK}. Finally, one can directly verify that the optimal controller constructed from the Riccati solutions in \Cref{theorem:seperation-controller} satisfies $\mathbf{R}_\mK(s)\equiv0$. Together with \Cref{theorem:global_optimality_RK}, this provides an alternative proof of the separation principle for LQG control. \hfill $\square$
\end{remark}

\section{Conclusion} \label{section:conclusion}

This paper has studied the nonconvex optimization landscape of LQG control under direct state-space parameterization. We have proved that every local minimum of the full-order LQG cost is globally optimal. More generally, at any controller order, every local minimum corresponding to an uncontrollable or unobservable realization attains the globally optimal LQG cost over controllers of arbitrary orders. We have further shown that every suboptimal stationary point of the full-order LQG cost can be converted into a strict saddle with probability one. 
Future interesting directions include 1) investigating whether similar results hold for other control problems (such as robust control and distributed control), and 2) developing local search algorithms that exploit the policy augmentation mechanism. 

\vspace{6pt}

\noindent \textbf{Statements on AI use}. The authors had the question of whether full-order LQG control admits suboptimal local minima in 2020, while preparing the manuscript \cite{tang2023analysis}. 
The initial proof for the main results in \Cref{section:main-results} was found by ChatGPT 5.6 Sol on July 15, 2026. One key idea in the proof was the link between recent results on the Hessian characterization \cite{zheng2022escaping} and the frequency-domain characterization \cite{li2025policy}. The initial proof by ChatGPT 5.6 Sol devoted substantial space to routine steps, with several genuinely nontrivial logical transitions treated too briefly. The authors subsequently verified, substantially revised, and rewrote the proofs, and take full responsibility for the accuracy and content of the final manuscript.

\bibliographystyle{unsrt}
\bibliography{ref.bib}

\newpage
\appendix

\section{Youla Parameterization and Proof of \Cref{theorem:global_optimality_RK}} \label{appendix:proof_global_optimality_RK}

As outlined after \Cref{theorem:global_optimality_RK}, the proof utilizes the convex reformulation of the LQG problem~\cref{eq:LQG-control} using Youla parameterization in the frequency domain. We present the proof details in this section: Given any stabilizing policy $\mK \in \mathcal{C}_q$, the Youla parameterization centered at $\mK$ represents all stabilizing policies, of arbitrary orders, through a stable Youla parameter $\mathbf{Q}$, with $\mathbf{Q} =0$ corresponding to $\mK$ (\Cref{appendix:Youla-parameterization}). Under this parameterization, the LQG problem \cref{eq:LQG-control} becomes a convex quadratic optimization problem in the Youla parameter $\mathbf{Q}$ (\Cref{appendix:convex-reformulation}). Consequently, $\mK$ is globally optimal if and only if the first-order derivative of the Youla-parameterized objective vanishes at $\mathbf{Q} =0$. We then show that this first-order condition is precisely equivalent to $\mathbf{R}_\mK(s) \equiv 0$ (\Cref{appendix:first-order-residual}). 

\vspace{-2mm}

\paragraph{Frequency-domain notation.} We extensively use frequency-domain representations. For completeness, we briefly review the frequency-domain spaces used below; see~\cite{zhou1996robust} for a classical textbook. We will only use continuous-time and real-rational transfer~matrices. 
Let $\mathcal{RL}_2^{a\times b}$ denote the space of strictly proper real-rational transfer matrices of size $a\times b$ with no poles on the imaginary axis.

For $\mathbf G,\mathbf H\in\mathcal{RL}_2^{a\times b}$, their $\mathcal L_2$ inner product and norm are defined by
\[
    \langle \mathbf G,\mathbf H\rangle_{\mathcal L_2}
    \coloneq
    \frac{1}{2\pi}\int_{-\infty}^{\infty}
    \operatorname{tr}\!\left(
        \mathbf G(j\omega)^*\mathbf H(j\omega)
    \right)\,d\omega,
    \qquad
    \|\mathbf G\|_{\mathcal L_2}^2
    \coloneq
    \langle \mathbf G,\mathbf G\rangle_{\mathcal L_2}.
\]
The space $\mathcal{RH}_2^{a\times b}\subset\mathcal{RL}_2^{a\times b}$ consists of $a\times b$ transfer matrices whose poles all lie in the open left half-plane, i.e., they are strictly proper and stable. For $\mathbf G\in\mathcal{RH}_2^{a\times b}$, we write $\|\mathbf G\|_{\mathcal H_2}\coloneq\|\mathbf G\|_{\mathcal L_2}$.
The space $\mathcal{RH}_\infty^{a\times b}$ consists of proper, stable,
real-rational transfer matrices, with
$
    \|\mathbf G\|_{\mathcal H_\infty}
    \coloneq
    \sup_{\omega\in\mathbb R}
    \sigma_{\max}\!\left(\mathbf G(j\omega)\right).
$ 
In particular, a nonzero constant transfer matrix of dimension $a \times b$ belongs to $\mathcal{RH}_\infty^{a\times b}$ but not to $\mathcal{RL}_2^{a\times b}$. The superscripts in $\mathcal{RL}_2^{a\times b}$, $\mathcal{RH}_2^{a\times b}$, and $\mathcal{RH}_\infty^{a\times b}$ will frequently be omitted when they can be inferred from the context.

Every $\mathbf G\in\mathcal{RL}_2$ admits a unique orthogonal stable--antistable decomposition
$
    \mathbf G=\mathbf G_- +\mathbf G_+,
$ 
where $\mathbf G_-$ has all its poles in the open right half-plane (i.e., all poles are antistable) and $\mathbf G_+\in\mathcal{RH}_2$. In particular, we have
$
    \langle \mathbf G_-,\mathbf H\rangle_{\mathcal L_2}=0,
    \,
    \forall\,\mathbf H\in\mathcal{RH}_2
$  \cite[Chapter 4.3]{zhou1996robust}. 
Finally, the \textit{paraconjugate transpose} of a real-rational transfer matrix $\mathbf G$ is defined by
$
    \mathbf G^{\sim}(s)\coloneq\mathbf G(-s)^\tr,
$ 
so that $\mathbf G^{\sim}(j\omega)=\mathbf G(j\omega)^*$.

\subsection{Youla parameterization of all stabilizing policies} \label{appendix:Youla-parameterization}

Fix \(q\in\mathbb{Z}_{\geq 0}\) and a stabilizing policy
\[
\mK=
\begin{bmatrix}
0_{m\times\ell}&C_{\mK}\\
B_{\mK}&A_{\mK}
\end{bmatrix}
\in\mathcal{C}_q.
\]
Starting from this nominal policy, we seek to parameterize stabilizing policies of arbitrary orders in \(\bigcup_{r\in\mathbb{Z}_{\geq 0}}\mathcal{C}_r\). For this purpose, we perturb the policy dynamics \cref{eq:dynamic_controller} as 
\begin{equation} \label{eq:nominal-policy-with-perturbation}
\begin{aligned}
    \dot{\xi}(t) ={} & A_{\mK}\xi(t)+B_\mK y(t) + \nu_\xi(t), \\
u(t) ={} & C_\mK \xi(t) + \nu_u(t), 
\end{aligned}
\end{equation}
where $\nu_\xi(t) \in \mathbb{R}^q,\nu_u(t) \in \mathbb{R}^m$ are perturbing signals. We let \(\nu=\operatorname{col}(\nu_u,\nu_\xi)\) depend dynamically on the system output $y$ and the controller state $\xi$, i.e., \(\eta=\operatorname{col}(y,\xi)\). Specifically, we introduce a finite-dimensional LTI system with transfer matrix $\boldsymbol{\Delta}$, and let $\nu$ be the output of $\boldsymbol{\Delta}$ with input $\eta$. 
In particular, if $\mathbf{\Delta}$ has state $\zeta(t)\in\mathbb{R}^p$, it admits a realization of the form
\begin{equation}
\label{eq:centered-youla-delta-realization}
\begin{aligned}
\dot{\zeta}(t)
&=A_\Delta\zeta(t)
+B_{\Delta y}y(t)
+B_{\Delta\xi}\xi(t),\\
\nu_u(t)
&=C_{\Delta u}\zeta(t)+D_{u\xi}\xi(t),\\
\nu_\xi(t)
&=C_{\Delta\xi}\zeta(t)
+D_{\xi y}y(t)
+D_{\xi\xi}\xi(t).
\end{aligned}
\end{equation}
The absence of a direct term from $y$ to $\nu_u$ in \cref{eq:centered-youla-delta-realization} ensures that the perturbed policy retains zero feedthrough from the measurement $y$ to the control input \(u\). \Cref{fig:perturbed_sys_Youla} shows the closed-loop system after the perturbation has been incorporated.

\begin{figure}[t]
\centering
\includegraphics[width=.45\linewidth]{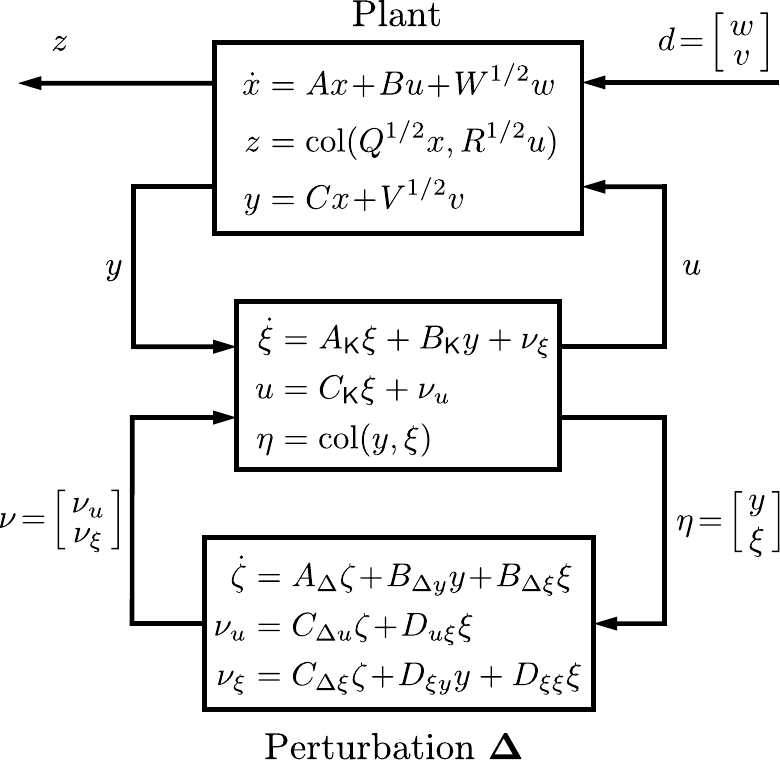}
\caption{The closed-loop system with perturbation $\boldsymbol{\Delta}$.}
\label{fig:perturbed_sys_Youla}
\end{figure}

\begin{figure}
\centering
\begin{subfigure}{.45\linewidth}
\centering
\includegraphics[width=\textwidth]{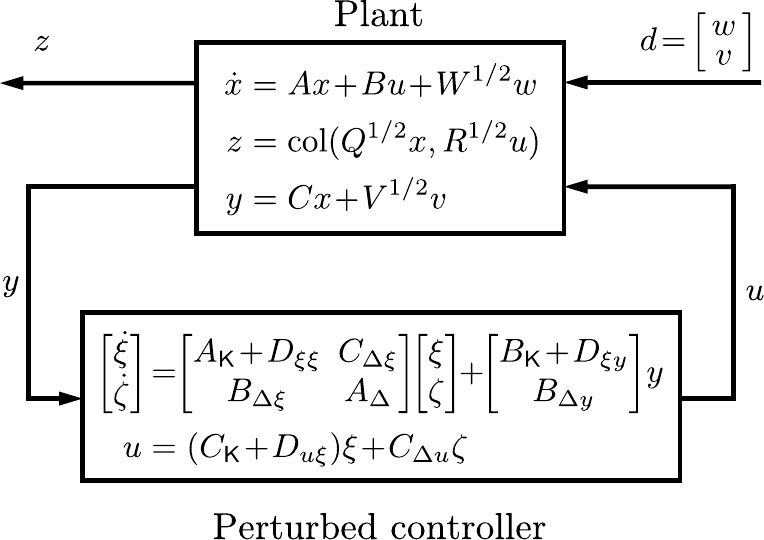}
\caption{}
\label{fig:perturbed_sys_Youla_1}
\end{subfigure}
\hfill
\begin{subfigure}{.45\linewidth}
\centering
\includegraphics[width=\textwidth]{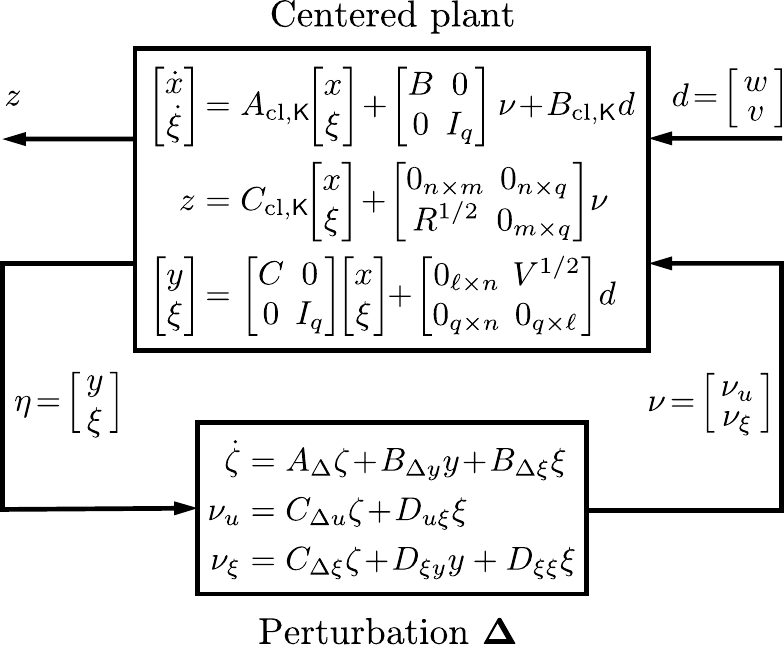}
\caption{}
\label{fig:perturbed_sys_Youla_2}
\end{subfigure}
\caption{Two equivalent formulations of the closed-loop system with perturbation $\boldsymbol{\Delta}$.}
\end{figure}

We may perceive the closed-loop system from two equivalent angles. The first angle is to combine \cref{eq:nominal-policy-with-perturbation} with \cref{eq:centered-youla-delta-realization} to form a perturbed controller first, and then connect it to the plant~\cref{eq:plant} to close the loop. 
Substituting \cref{eq:centered-youla-delta-realization} into \cref{eq:nominal-policy-with-perturbation} yields the perturbed controller with state \(\operatorname{col}(\xi,\zeta)\) and realization
\begin{equation}
\label{eq:centered-youla-patched-controller}
\mK_\Delta=
\left[
\begin{array}{c:cc}
0 &C_{\mK}+D_{u\xi}&C_{\Delta u}\\ \hdashline
B_{\mK}+D_{\xi y} &A_{\mK}+D_{\xi\xi}&C_{\Delta\xi}\\
B_{\Delta y} &B_{\Delta\xi}&A_\Delta
\end{array}
\right].
\end{equation}
Also see \cref{fig:perturbed_sys_Youla_1} for an illustration.
Thus, a static perturbation (\(p=0\)) modifies \((A_{\mK},B_{\mK},C_{\mK})\) without changing the policy order, and a dynamic perturbation with \(p>0\) states produces an order-\((q+p)\) realization.
We say that $\boldsymbol{\Delta}$ is an \textit{admissible internally stabilizing perturbation}, if $\mK_\Delta\in\mathcal{C}_{p+q}$ for any minimal $p$-dimensional realization of $\boldsymbol{\Delta}$ with the form~\cref{eq:centered-youla-delta-realization}.

The second angle is to combine the original plant \cref{eq:plant} with \cref{eq:nominal-policy-with-perturbation} to form a ``centered'' plant, with the perturbation $\boldsymbol{\Delta}$ now serving as a feedback controller for the centered plant. Specifically,
we substitute \cref{eq:nominal-policy-with-perturbation} into the~plant
\cref{eq:plant}, which yields the centered plant with inputs
\((d,\nu)\) and outputs \((z,\eta)\). The output $\eta$ will then be fed into $\boldsymbol{\Delta}$, producing $\nu$ that will be fed back to the centered plant. See \cref{fig:perturbed_sys_Youla_2} for an illustration of the second angle. Some straightforward computations then show that the transfer matrix of the centered plant from $(d,\nu)$ to $(z,\eta)$ is given by
\begin{equation}
\label{eq:centered-youla-four-block-map}
\mathbf{P}_{\mK}^{\mathrm{ctr}}=
\begin{bmatrix}
\mathbf{M}_{11}&\mathbf{M}_{12}\\
\mathbf{M}_{21}&\mathbf{M}_{22}
\end{bmatrix},
\end{equation}
where 
\begin{subequations} \label{eq:centered-youla-M}
\begin{align}
\mathbf{M}_{11}(s)
&=C_{\mathrm{cl},\mK}(sI-A_{\mathrm{cl},\mK})^{-1}B_{\mathrm{cl},\mK},\label{eq:centered-youla-M11}\\
\mathbf{M}_{12}(s)
&=C_{\mathrm{cl},\mK}(sI-A_{\mathrm{cl},\mK})^{-1}\mathfrak{B}+\mathfrak{D}_{12},
\label{eq:centered-youla-M12}\\
\mathbf{M}_{21}(s)
&=\mathfrak{C}(sI-A_{\mathrm{cl},\mK})^{-1}B_{\mathrm{cl},\mK}+\mathfrak{D}_{21},
\label{eq:centered-youla-M21}\\
\mathbf{M}_{22}(s)
&=\mathfrak{C}(sI-A_{\mathrm{cl},\mK})^{-1}\mathfrak{B}.
\label{eq:centered-youla-M22}
\end{align}
\end{subequations}
with  
\begin{equation} \label{eq:augmented-matrices}
\mathfrak{B}=
\begin{bmatrix}
B&0_{n\times q}\\
0_{q\times m}&I_q
\end{bmatrix},
\,
\mathfrak{C}=
\begin{bmatrix}
C&0_{\ell\times q}\\
0_{q\times n}&I_q
\end{bmatrix},
\, 
\mathfrak{D}_{12}=
\begin{bmatrix}
0_{n\times m}&0_{n\times q}\\
R^{1/2}&0_{m\times q}
\end{bmatrix},
\,
\mathfrak{D}_{21}=
\begin{bmatrix}
0_{\ell\times n}&V^{1/2}\\
0_{q\times n}&0_{q\times\ell}
\end{bmatrix}.
\end{equation}
Since the nominal policy $\mK$ is internally stabilizing, the centered plant is open-loop internally stable.

The closed-loop transfer matrix from $d$ to $z$ is then given by
\begin{equation}
    \mathbf{T}_{zd,\mK_{\Delta}}(s) = \mathbf{M}_{11}(s) + \mathbf{M}_{12}(s)\mathbf{\Delta}(s)(I - \mathbf{M}_{22}(s)\mathbf{\Delta}(s))^{-1}\mathbf{M}_{21}(s). 
\end{equation}
It turns out that the perturbed policy $\mK_{\Delta}$ in \cref{eq:centered-youla-patched-controller} internally stabilizes the plant \cref{eq:plant} if and only if $\mathbf{\Delta}(I - \mathbf{M}_{22}\mathbf{\Delta})^{-1}$ is a stable transfer matrix. Moreover, any stabilizing policy of arbitrary order can be represented by a suitable stable transfer matrix $\mathbf{Q} = \mathbf{\Delta}(I - \mathbf{M}_{22}\mathbf{\Delta})^{-1}$. This stable transfer matrix $\mathbf{Q}$ is called the Youla parameter \cite{youla1976modern}. 

The policy structure in \cref{eq:centered-youla-delta-realization,eq:dynamic_controller}
requires the direct feedthrough from \(y\) to \(u\) to vanish. We therefore define the admissible feedthrough subspace
\begin{equation}
\label{eq:centered-youla-feedthrough-space}
\mathcal{D}_q\coloneq
\left\{
\begin{bmatrix}
0_{m\times\ell}&D_{u\xi}\\
D_{\xi y}&D_{\xi\xi}
\end{bmatrix} \in \mathbb{R}^{(m + q)\times(\ell + q)}
: 
D_{u\xi}\in\mathbb{R}^{m\times q},
D_{\xi y}\in\mathbb{R}^{q\times\ell},
D_{\xi\xi}\in\mathbb{R}^{q\times q}
\right\},
\end{equation}
and the Youla space
\begin{equation}
\label{eq:centered-youla-space}
\mathcal{Y}_q\coloneq
\left\{
\mathbf{Q}\in
\mathcal{RH}_\infty^{(m+q)\times(\ell+q)}
:
\mathbf{Q}(\infty)\in\mathcal{D}_q
\right\}.
\end{equation}
It is clear that every \(\mathbf{Q}\in\mathcal{Y}_q\) admits the unique decomposition
$
\mathbf{Q}=\mathbf{Q}_0+D,
\,
\mathbf{Q}_0\in
\mathcal{RH}_{2},
\, 
D\in\mathcal{D}_q.
$ 
In particular, \(\mathcal{Y}_q\) is a real vector space. 

We have the following characterization of internally stabilizing policies of arbitrary order. 

\begin{lemma}[Centered Youla parameterization]
\label{lemma:centered-youla-parameterization}
Fix \(q\in\mathbb Z_{\geq0}\) and \(\mK\in\mathcal C_q\). Let \(\boldsymbol{\Delta}\) be a proper
real-rational transfer matrix satisfying \(\boldsymbol{\Delta}(\infty)\in\mathcal D_q\), with a minimal \(p\)-dimensional realization of the form \cref{eq:centered-youla-delta-realization}. Let \(\mK_\Delta\) be the corresponding perturbed policy in \cref{eq:centered-youla-patched-controller}, and define
\begin{equation}
\label{eq:centered-youla-forward-map}
\mathbf Q
\coloneq
\boldsymbol{\Delta}
\bigl(I-\mathbf M_{22}\boldsymbol{\Delta}\bigr)^{-1},
\end{equation}
where \(\mathbf M_{22}\) is defined in
\cref{eq:centered-youla-M22}. Then:
\begin{enumerate}
\item The map \(\boldsymbol{\Delta}\mapsto\mathbf Q\) in
\cref{eq:centered-youla-forward-map} is a bijection from the set of admissible
internally stabilizing perturbations onto
\(\mathcal Y_q\), with inverse
\begin{equation}
\label{eq:centered-youla-inverse-map}
\boldsymbol{\Delta}
=
\bigl(I+\mathbf Q\mathbf M_{22}\bigr)^{-1}\mathbf Q
=
\mathbf Q\bigl(I+\mathbf M_{22}\mathbf Q\bigr)^{-1}.
\end{equation}
In particular, every \(\mathbf Q\in\mathcal Y_q\) determines a unique proper real-rational perturbation $\boldsymbol{\Delta}$ satisfying \(\boldsymbol{\Delta}(\infty)\in\mathcal D_q\), and any minimal realization of this perturbation yields an internally stabilizing perturbed policy $\mK_\Delta$. 
\item For every \(r\in\mathbb Z_{\geq0}\) and every
\(\widehat{\mK}\in\mathcal C_r\), there exists
\(\mathbf Q\in\mathcal Y_q\) such that the perturbed policy obtained
from \cref{eq:centered-youla-inverse-map} has the same controller
transfer matrix as \(\widehat{\mK}\); namely,
\begin{equation}
\label{eq:recover_arbitrary_controller}
C_{\mK_\Delta}(sI-A_{\mK_\Delta})^{-1}B_{\mK_\Delta}
=
C_{\widehat{\mK}}
(sI-A_{\widehat{\mK}})^{-1}B_{\widehat{\mK}}.
\end{equation}
\end{enumerate}
\end{lemma}

\begin{proof}
We prove the two statements below. 
\begin{enumerate}
\item We adopt the second angle shown in \cref{fig:perturbed_sys_Youla_2}. Since $\mK\in\mathcal{C}_q$, the centered plant is internally stable. It is then a direct consequence of \cite[Theorem~12.7]{zhou1996robust} that (any minimal realization of) $\boldsymbol{\Delta}$ internally stabilizes the centered plant if and only if $\boldsymbol{\Delta}(I-\mathbf{M}_{22}\boldsymbol{\Delta})^{-1}\in\mathcal{RH}_\infty$. Then $\mathbf{Q}(\infty)=\boldsymbol{\Delta}(\infty)
\bigl(I-\mathbf M_{22}(\infty)\boldsymbol{\Delta}(\infty)\bigr)^{-1}=\boldsymbol{\Delta}(\infty)$ indicates that $\mathbf{Q}(\infty)\in\mathcal{D}_q$ if and only if $\boldsymbol{\Delta}(\infty)\in\mathcal{D}_q$. The proof of the first statement is now complete.

\item Let $r$ and $\widehat{\mK}=\begin{bmatrix}
0 & C_{\widehat\mK} \\ B_{\widehat\mK} & A_{\widehat\mK}
\end{bmatrix}\in\mathcal{C}_r$ be arbitrary. Without loss of generality, we may assume that $(A_{\widehat\mK},B_{\widehat\mK},C_{\widehat\mK})$ is minimal. Choose
\[
\begin{aligned}
\boldsymbol{\Delta}(s) ={} & \begin{bmatrix}
C_{\widehat\mK}(sI-A_{\widehat\mK})^{-1}B_{\widehat\mK}&-C_{\mK}\\
-B_\mK &-(A_\mK+\alpha I_q)
\end{bmatrix} \\
={} & \begin{bmatrix}
C_{\widehat\mK} \\
0
\end{bmatrix}
(sI-A_{\widehat\mK})^{-1}
\begin{bmatrix}
B_{\widehat\mK} & 0
\end{bmatrix} + \begin{bmatrix}
0&-C_{\mK}\\
-B_\mK &-(A_\mK+\alpha I_q)
\end{bmatrix},
\end{aligned}
\]
where $\alpha>0$ is arbitrary. A state-space realization of $\boldsymbol{\Delta}$ is given by
\[
\begin{aligned}
\dot{\zeta}(t) ={} & A_{\widehat\mK}\zeta(t) + B_{\widehat{\mK}}y(t) \\
\nu_u(t) ={} & C_{\widehat{\mK}}\zeta(t) - C_\mK\xi(t), \\
\nu_\xi(t) ={} &
-B_\mK y(t) - (A_\mK+\alpha I_q)\xi(t),
\end{aligned}
\]
and this realization is clearly minimal.
By~\cref{eq:centered-youla-patched-controller}, the perturbed controller is then
\[
\mK_\Delta = \left[
\begin{array}{c:cc}
0 & 0 &C_{\widehat\mK}\\ \hdashline
0 &-\alpha I_q & 0 \\
B_{\widehat\mK} & 0 & A_{\widehat{\mK}}
\end{array}
\right].
\]
Since $\widehat\mK\in\mathcal{C}_r$ and $\alpha>0$, the closed-loop system with the perturbed controller $\mK_\Delta$ is internally stable. Therefore, $\boldsymbol{\Delta}$ is an admissible internally stabilizing perturbation, and by the first part of the lemma, we have $\mathbf{Q}=\boldsymbol{\Delta}
\bigl(I-\mathbf M_{22}\boldsymbol{\Delta}\bigr)^{-1}\in\mathcal{Y}_q$. The identity~\cref{eq:recover_arbitrary_controller} directly follows from the expression of $\mK_\Delta$.
\qedhere
\end{enumerate}
\end{proof}

\subsection{Equivalent convex reformulation in the frequency domain}
\label{appendix:convex-reformulation}

Fix \(q\in\mathbb Z_{\geq0}\) and any stabilizing policy \(\mK\in\mathcal C_q\). By \cref{lemma:centered-youla-parameterization}, every \(\mathbf Q\in\mathcal Y_q\) determines an internally stabilizing perturbed policy, while every stabilizing policy of arbitrary order is represented at the input-output level by some \(\mathbf Q\in\mathcal Y_q\). Moreover, the corresponding closed-loop transfer matrix~is
\[
\mathbf T_{zd}(\mathbf Q)
=
\mathbf M_{11}
+\mathbf M_{12}\mathbf Q\mathbf M_{21}.
\]
Consequently, the LQG problem \cref{eq:LQG-control} admits the
equivalent reformulation
\begin{equation}
\label{eq:centered-youla-convex-reformulation}
J^\star
=
\min_{\mathbf Q\in\mathcal Y_q}
\Phi_{\mK}(\mathbf Q),
\qquad
\Phi_{\mK}(\mathbf Q)
\coloneq
\left\|
\mathbf M_{11}
+\mathbf M_{12}\mathbf Q\mathbf M_{21}
\right\|_{\mathcal H_2}^2.
\end{equation}

The nominal policy corresponds to \(\mathbf Q=0\), with \(\Phi_{\mK}(0)=J_q(\mK)\). As expected, the objective in \cref{eq:centered-youla-convex-reformulation} is well defined for every \(\mathbf Q\in\mathcal Y_q\). Indeed, \(\mathbf M_{12}\mathbf Q\mathbf M_{21}\) is stable, and its direct feedthrough is
\[
\mathfrak{D}_{12}\mathbf Q(\infty)\mathfrak{D}_{21}=0
\]
because of the zero structures of $\mathfrak{D}_{12}$ and $\mathfrak{D}_{21}$  in \cref{eq:augmented-matrices} and \(\mathbf Q(\infty)\in\mathcal D_q\). Hence, the closed-loop transfer matrix \( \mathbf T_{zd}(\mathbf Q) = \mathbf M_{11}+\mathbf M_{12}\mathbf Q\mathbf M_{21}\) is strictly proper, belongs to $\mathcal{RH}_2$ and thus has a finite \(\mathcal H_2\)-norm for~any~$\mathbf Q\in\mathcal Y_q$. 

Note that $\Phi_{\mK}(\mathbf Q)$ in \cref{eq:centered-youla-convex-reformulation} can be viewed as a convex quadratic functional in $\mathbf{Q}$. From this frequency-domain representation, $J_q(\mK)=\Phi_{\mK}(0)$ achieves the globally optimal value if and only if the first-order derivative of the Youla-parameterized objective vanishes at $\mathbf{Q} =0$. Let 
\begin{equation} \label{eq:first-order-residual-Q}
    \mathcal{L}_{\mK}(\mathbf{Q})\coloneq \langle \mathbf{M}_{11}, \mathbf{M}_{12}\mathbf{Q}\mathbf{M}_{21} \rangle_{\mathcal{H}_2}
\end{equation}
denotes the first-order residual of $\Phi_{\mK}(\mathbf Q)$ at $\mathbf{Q} = 0$. 
We have the following result.

\begin{lemma} \label{lemma:first-order-residual=0}
    Let $q\in\mathbb{Z}_{\geq 0}$ satisfy $\mathcal{C}_q\neq\varnothing$. Then $\mK \in\mathcal{C}_{q}$ is a globally optimal policy, i.e., $J_q(\mK) = J^\star$, if and only if 
    \begin{equation} \label{eq:first-order-residual-zero}
        \mathcal{L}_{\mK}(\mathbf{Q}) = 0, \qquad \forall \mathbf{Q} \in \mathcal Y_q. 
    \end{equation}
\end{lemma}

\begin{proof}
    By \cref{lemma:centered-youla-parameterization}, for every $\widehat{\mK} \in \mathcal{C}_r$ of arbitrary order $r$, we can find some \(\mathbf Q\in\mathcal Y_q\) such that $J_r(\widehat{\mK}) = \Phi_{\mK}(\mathbf Q)$. Therefore,
    $$
    \begin{aligned}
J_r(\widehat{\mK}) - J_q(\mK) &= \Phi_{\mK}(\mathbf Q) - \Phi_{\mK}(0) \\
      &= 2\mathcal{L}_{\mK}(\mathbf{Q}) + \left\| \mathbf M_{12}\mathbf Q\mathbf M_{21}
\right\|_{\mathcal H_2}^2.
    \end{aligned}
    $$
    If \cref{eq:first-order-residual-zero} holds, then we have $J_r(\widehat{\mK}) - J_q(\mK)\geq 0$. By the arbitrariness of $r$ and $\widehat{\mK}\in\mathcal{C}_r$, we see that $\mK$ is globally optimal.

    Conversely, suppose $\mK$ is globally optimal. Let $\mathbf{Q}\in \mathcal Y_q$ be arbitrary. Since $ \mathcal Y_q$ is a vector space, $t\mathbf{Q}\in \mathcal Y_q$ for any $t \in \mathbb{R}$. Therefore, we have 
    $$
    \Phi_{\mK}(t\mathbf Q) - \Phi_{\mK}(0) = 2t\mathcal{L}_{\mK}(\mathbf{Q}) + t^2  \left\| \mathbf M_{12}\mathbf Q\mathbf M_{21}
\right\|_{\mathcal H_2}^2 \geq 0, \qquad \forall t \in \mathbb{R}. 
    $$
    The first-order term must vanish, so that $\mathcal{L}_{\mK}(\mathbf{Q}) = 0$. Since $\mathbf{Q}\in \mathcal Y_q$ is arbitrary, we arrive at the desired claim \cref{eq:first-order-residual-zero}. 
\end{proof}

\Cref{lemma:first-order-residual=0} is a direct consequence of the frequency-domain representation \cref{eq:centered-youla-convex-reformulation}. Thanks to the convexity of $\Phi_{\mK}$ in terms of the Youla parameter $\mathbf{Q}$, we get the global optimality condition \cref{eq:first-order-residual-zero}. Note that $\mathcal{L}_{\mK}(\mathbf{Q})$ in \cref{eq:first-order-residual-Q} is a linear functional of $\mathbf{Q}$. Furthermore, we can use the cyclic property of the trace operator to rewrite it into the following form   
\begin{equation} \label{eq:first-order-residual-cyclic}
\begin{aligned}
    \mathcal{L}_{\mK}(\mathbf{Q}) &= \frac{1}{2\pi} \int_{-\infty}^{\infty} \operatorname{tr}\left(\mathbf{M}_{11}(j\omega)^*\mathbf{M}_{12}(j\omega)\mathbf{Q}(j\omega)\mathbf{M}_{21}(j\omega)\right)d\omega \\
    &= \frac{1}{2\pi} \int_{-\infty}^{\infty} \operatorname{tr}\left(\mathbf{M}_{21}(j\omega)\mathbf{M}_{11}(j\omega)^*\mathbf{M}_{12}(j\omega)\mathbf{Q}(j\omega)\right)d\omega \\ 
    &= \frac{1}{2\pi} \int_{-\infty}^{\infty} \operatorname{tr}\left(\mathbf \Psi_{\mK}(j\omega)^*\mathbf{Q}(j\omega)\right) d\omega,
\end{aligned}
\end{equation}
where we denote $\mathbf{\Psi}_{\mK}(j\omega) = \mathbf{M}_{12}(j\omega)^*\mathbf{M}_{11}(j\omega)\mathbf{M}_{21}(j\omega)^*$.

By the conjugate of complex matrices, we have  $H(j\omega)^* = H(-j\omega)^\tr$. 
Recall that for any real-rational transfer matrix $\mathbf{G}$,  its \textit{paraconjugate transpose} is denoted by $\mathbf{G}^{\sim}(s) \coloneq \mathbf{G}(-s)^{\top}$. We can then define 
\begin{equation} \label{eq:paraconjugate-Psi}
    \mathbf{\Psi}_{\mK}(s) = \mathbf{M}_{12}^{\sim}(s)\mathbf{M}_{11}(s)\mathbf{M}_{21}^{\sim}(s)
\end{equation}
which is consistent with $\mathbf{\Psi}_{\mK}(j\omega) = \mathbf{M}_{12}(-j\omega)^\tr \mathbf{M}_{11}(j\omega)\mathbf{M}_{21}(-j\omega)^\tr = \mathbf{M}_{12}(j\omega)^*\mathbf{M}_{11}(j\omega)\mathbf{M}_{21}(j\omega)^*$. 
Even though $ \mathbf{M}_{12}, \mathbf{M}_{11}, \mathbf{M}_{21}$ in \cref{eq:centered-youla-M} are all stable given $\mK \in \mathcal{C}_q$, the new transfer matrix $\mathbf{\Psi}_{\mK}(s)$ in \cref{eq:paraconjugate-Psi} generally has unstable components due to the {paraconjugate transpose} in its definition. We only have $\boldsymbol{\Psi}_{\mK}\in\mathcal{RL}_2$, but it is in general not in $\mathcal{RH}_2$. 

With this definition, for every $\mathbf{Q}_0\in\mathcal{RH}_2$,
we may rewrite \cref{eq:first-order-residual-cyclic} as
\[
    \mathcal{L}_{\mK}(\mathbf{Q}_0)
    =
    \left\langle
        \boldsymbol{\Psi}_{\mK},\mathbf{Q}_0
    \right\rangle_{\mathcal{L}_2}.
\]
Notice, however, that a general $\mathbf{Q}\in\mathcal{Y}_q$ may have nonzero~direct feedthrough and thus need not belong to $\mathcal{RL}_2$. For such $\mathbf{Q}$, the right-hand side of \cref{eq:first-order-residual-cyclic} is understood through its explicit integral representation. This integral is still well defined because $\mathbf{M}_{11}$ and $\mathbf{M}_{12}\mathbf{Q}\mathbf{M}_{21}$ both belong to $\mathcal{RH}_2$. Consequently, by \Cref{lemma:first-order-residual=0},
$\mK\in\mathcal{C}_q$ is globally optimal if and only if
\[
    \frac{1}{2\pi}\int_{-\infty}^{\infty}
    \operatorname{tr}\!\left(
        \boldsymbol{\Psi}_{\mK}(j\omega)^*
        \mathbf{Q}(j\omega)
    \right)\,d\omega
    =0,
    \qquad
    \forall\,\mathbf{Q}\in\mathcal{Y}_q.
\]

\subsection{Proof of \Cref{theorem:global_optimality_RK}} \label{appendix:first-order-residual}

Our last technical ingredient to establish \Cref{theorem:global_optimality_RK} is an explicit stable/antistable decomposition for $\mathbf{\Psi}_{\mK}(s)$, which implies that 
\begin{equation}
    \langle \mathbf{\Psi}_{\mK}, \mathbf{Q}_0 \rangle_{\mathcal{L}_2} = \langle \mathbf{R}_{\mK}, \mathbf{Q}_0 \rangle_{\mathcal{H}_2}, \qquad \forall \mathbf{Q}_0 \in \mathcal{RH}_2, 
\end{equation}
where $\mathbf{R}_{\mK}$, defined in \Cref{theorem:global_optimality_RK}, turns to be the stable component of $\mathbf{\Psi}_{\mK}$. In particular, we have the following result.

\begin{lemma} \label{lemma:stable-decomposition-Psi}
     Let $q\in\mathbb{Z}_{\geq 0}$ satisfy $\mathcal{C}_q\neq\varnothing$, and let $\mK \in\mathcal{C}_{q}$ be arbitrary. The transfer matrix $\mathbf{\Psi}_{\mK}(s)$ in \cref{eq:paraconjugate-Psi} can be decomposed into a stable component and an antistable component as 
     \begin{equation}
         \mathbf{\Psi}_{\mK}(s) = \mathbf{\Psi}_{\mK,-}(s) + \mathbf{R}_{\mK}(s),
     \end{equation}
     where all poles of $\mathbf{\Psi}_{\mK,-}(s)$ have positive real parts and $\mathbf{R}_{\mK}(s)$ is defined in \Cref{theorem:global_optimality_RK} with all poles having negative real parts. Furthermore, both  $\mathbf{\Psi}_{\mK,-}(s)$ and $ \mathbf{R}_{\mK}(s)$ are strictly proper. 
\end{lemma}

The proof is based on a direct state-space computation, and we provide it at the end of this subsection. With \Cref{lemma:stable-decomposition-Psi}, we are now ready to prove  \Cref{theorem:global_optimality_RK}.

\begin{proof}[Proof of \Cref{theorem:global_optimality_RK}]
Suppose $\mK \in \mathcal{C}_q$ is globally optimal, we have $\mathcal{L}_{\mK}(\mathbf{Q}) = 0, \forall \mathbf{Q} \in \mathcal{Y}_q$ by \Cref{lemma:first-order-residual=0}. In particular, for any strictly proper and stable $\mathbf{Q}_0$, we have 
\begin{equation} \label{eq:stable-anti-stable-decomposition-Q}
\begin{aligned}
\mathcal{L}_{\mK}(\mathbf{Q}_0) = \langle \mathbf{\Psi}_{\mK}, \mathbf{Q}_0 \rangle_{\mathcal{L}_2} &= \langle \mathbf{\Psi}_{\mK,-} + \mathbf{R}_{\mK}, \mathbf{Q}_0 \rangle_{\mathcal{L}_2}  \\
&=\langle \mathbf{R}_{\mK}, \mathbf{Q}_0 \rangle_{\mathcal{H}_2}, 
\end{aligned}
\end{equation}
where the second equality is from \Cref{lemma:stable-decomposition-Psi} and the third equality is a standard result that $ \langle \mathbf{\Psi}_{\mK,-}, \mathbf{Q}_0 \rangle_{\mathcal{L}_2} = 0$ since $\mathbf{\Psi}_{\mK,-}$ is orthogonal to any $\mathbf{Q}_0 \in \mathcal{RH}_2$. Now, choosing $\mathbf{Q}_0 = \mathbf{R}_{\mK}$, which belongs to $\mathcal{RH}_2$ by definition, we get 
$$
0 = \langle \mathbf{R}_{\mK}, \mathbf{R}_{\mK} \rangle_{\mathcal{H}_2} = \|\mathbf{R}_{\mK}\|_{\mathcal{H}_2}^2,
$$
which guarantees that $\mathbf{R}_{\mK}(s) \equiv 0$.  

Conversely, suppose $\mathbf{R}_{\mK}(s) \equiv 0$. From \Cref{eq:stable-anti-stable-decomposition-Q}, we naturally have $\mathcal{L}_{\mK}(\mathbf{Q}_0) = 0$ for any $\mathbf{Q}_0 \in \mathcal{RH}_{2}$. 
For any $\mathbf{Q} \in \mathcal{Y}_q$, we can decompose it as 
$
\mathbf{Q} = \mathbf{Q}_0 + D
$
where $\mathbf{Q}_0\in \mathcal{RH}_{2}$ and $D\in\mathcal{D}_q$ is a constant matrix with a special zero block in \cref{eq:centered-youla-feedthrough-space}. Then, thanks to linearity, we have 
$$
\mathcal{L}_{\mK}(\mathbf{Q}) = \mathcal{L}_{\mK}(\mathbf{Q}_0) + \mathcal{L}_{\mK}(D) =  \mathcal{L}_{\mK}(D),
$$
since $\mathcal{L}_{\mK}(\mathbf{Q}_0) = 0$. 
It remains to prove that $\mathcal{L}_{\mK}(D) = 0$ for all $D\in\mathcal{D}_q$.

Fix any ${D}\in\mathcal{D}_q$. For each $\alpha>0$, define
\[
    \mathbf{D}_{\alpha}(s)
    \coloneq
    \frac{\alpha}{s+\alpha}{D}.
\]
Then $\mathbf{D}_{\alpha}\in\mathcal{RH}_2$ by definition, and hence
$
    \mathcal{L}_{\mK}(\mathbf{D}_{\alpha})=0.
$
We then have
\[
\begin{aligned}
\mathcal{L}_{\mK}(D)
={} & \frac{1}{2\pi}\int_{-\infty}^{+\infty}
\operatorname{tr}\left(\mathbf{M}_{11}(j\omega)^*\mathbf{M}_{12}(j\omega)D\mathbf{M}_{21}(j\omega)\right) d\omega \\
={} &
\frac{1}{2\pi}\int_{-\infty}^{+\infty}
\operatorname{tr}\left(\mathbf{M}_{11}(j\omega)^*\mathbf{M}_{12}(j\omega)\lim_{\alpha\rightarrow\infty}\mathbf{D}_\alpha(j\omega)\,\mathbf{M}_{21}(j\omega)\right) d\omega \\
={} & \lim_{\alpha\rightarrow\infty}\frac{1}{2\pi}\int_{-\infty}^{+\infty}
\operatorname{tr}\left(\mathbf{M}_{11}(j\omega)^*\mathbf{M}_{12}(j\omega)\mathbf{D}_\alpha(j\omega)\mathbf{M}_{21}(j\omega)\right) d\omega \\
={} &
\lim_{\alpha\rightarrow\infty}
\mathcal{L}_{\mK}(\mathbf{D}_\alpha) = 0,
\end{aligned}
\]
where the third equality can be justified by the dominated convergence theorem. The proof is now complete. \qedhere

\end{proof}

The main proof idea is that $\mathbf R_{\mK}$ is precisely the stable component of $\boldsymbol{\Psi}_{\mK}$. If the admissible Youla space $\mathcal{Y}_q$ contained only strictly proper stable transfer matrices, the result would follow immediately from the orthogonality between stable and
antistable transfer matrices: Testing against $\mathbf Q_0\in\mathcal{RH}_2$ eliminates the antistable component $\boldsymbol{\Psi}_{\mK,-}$ and leaves only $\mathbf R_{\mK}$.
The additional argument is needed because $\mathcal Y_q$ also contains
admissible constant feedthrough terms $D\in\mathcal D_q$. A nonzero constant transfer matrix does not belong to the continuous-time $\mathcal L_2$ space, and so the standard stable--antistable orthogonality cannot be applied to $D\in\mathcal D_q$ directly. The approximation $\mathbf{D}_\alpha(s)=\frac{\alpha}{s+\alpha} D$ bridges this gap.

We finally prove the orthogonal decomposition in \Cref{lemma:stable-decomposition-Psi}.

\begin{proof}[Proof of \Cref{lemma:stable-decomposition-Psi}]
    We first derive an explicit state-space realization for the transfer matrix $\mathbf{\Psi}_{\mK}(s)$ in \cref{eq:paraconjugate-Psi}. Note that given a transfer matrix $\mathbf{G}(s) = C(sI - A)^{-1}B + D$, its {paraconjugate transpose} can be written as $\mathbf{G}^{\sim}(s) = -B^\tr (sI + A^\tr )^{-1}C^\tr + D^\tr$. Recalling the definition \cref{eq:centered-youla-M}, we have  
    $$
    \begin{aligned}
    \mathbf{M}_{12}^{\sim}(s)
&=(-\mathfrak{B}^\tr)(sI+A_{\mathrm{cl},\mK}^\tr)^{-1}C_{\mathrm{cl},\mK}^\tr+\mathfrak{D}_{12}^\tr,
\\
\mathbf{M}_{21}^{\sim}(s)
&=B_{\mathrm{cl},\mK}^\tr(sI+A_{\mathrm{cl},\mK}^\tr)^{-1}(-\mathfrak{C}^\tr)+\mathfrak{D}_{21}^\tr.
\end{aligned}
    $$
In other words, using the standard state-space notation
$
    \left[
    \begin{array}{c|c}
        A & B\\ \hline
        C & D
    \end{array}
    \right]
    \coloneq
    C(sI-A)^{-1}B+D,
$ 
we have the following state-space realizations
\begin{equation*}
\def\arraystretch{1.25}
\begin{aligned}
    \mathbf{M}_{12}^{\sim}(s)
    &=
    \left[
    \begin{array}{c|c}
        -A_{\mathrm{cl},\mK}^{\tr}
        & C_{\mathrm{cl},\mK}^{\tr}\\ \hline
        -\mathfrak{B}^{\tr}
        & \mathfrak{D}_{12}^{\tr}
    \end{array}
    \right], \;
    \mathbf{M}_{11}(s)
    =
    \left[
    \begin{array}{c|c}
        A_{\mathrm{cl},\mK}
        & B_{\mathrm{cl},\mK}\\ \hline
        C_{\mathrm{cl},\mK}
        & 0
    \end{array}
    \right], \;
    \mathbf{M}_{21}^{\sim}(s)
    =
    \left[
    \begin{array}{c|c}
        -A_{\mathrm{cl},\mK}^{\tr}
        & -\mathfrak{C}^{\tr}\\ \hline
        B_{\mathrm{cl},\mK}^{\tr}
        & \mathfrak{D}_{21}^{\tr}
    \end{array}
    \right].
\end{aligned}
\end{equation*}
Consequently, using the standard cascade operation \cite[Section~3.6]{zhou1996robust}, the transfer matrix $\boldsymbol{\Psi}_{\mK} =\mathbf{M}_{12}^{\sim}\mathbf{M}_{11} \mathbf{M}_{21}^{\sim}$ has the following state-space realization
\begin{equation}
\def\arraystretch{1.25}
\boldsymbol{\Psi}_{\mK}(s)
=
\left[
\begin{array}{ccc|c}
    -A_{\mathrm{cl},\mK}^{\tr}
    & C_{\mathrm{cl},\mK}^{\tr}C_{\mathrm{cl},\mK}
    & 0
    & 0\\
    0
    & A_{\mathrm{cl},\mK}
    & B_{\mathrm{cl},\mK}B_{\mathrm{cl},\mK}^{\tr}
    & B_{\mathrm{cl},\mK}\mathfrak{D}_{21}^{\tr}\\
    0
    & 0
    & -A_{\mathrm{cl},\mK}^{\tr}
    & -\mathfrak{C}^{\tr}\\ \hline
    -\mathfrak{B}^{\tr}
    & \mathfrak{D}_{12}^{\tr}C_{\mathrm{cl},\mK}
    & 0
    & 0
\end{array}
\right]
\label{eq:Psi-cascade-realization}
\end{equation}
Since $\mK\in\mathcal C_q$, the matrix $A_{\mathrm{cl},\mK}$ is Hurwitz. Thus, the realization
in \cref{eq:Psi-cascade-realization} contains stable eigenvalues from $A_{\mathrm{cl},\mK}$ and antistable eigenvalues from $-A_{\mathrm{cl},\mK}^{\tr}$. To separate these two parts, let
$\hat{x}$ denote the state of the realization in \cref{eq:Psi-cascade-realization} and introduce the coordinate transformation
\begin{equation}
    z=\mathcal T_{\mK}\hat{x},
    \qquad
    \mathcal T_{\mK}
    \coloneq
    \begin{bmatrix}
        I & Y_{\mK} & 0\\
        0 & I & -X_{\mK}\\
        0 & 0 & I
    \end{bmatrix}, 
\label{eq:Psi-coordinate-transformation}
\end{equation}
where $X_\mK$ and $Y_\mK$ are the unique solutions to the Lyapunov equations \cref{eq:closed-loop_Lyapunov}. The inverse of $\mathcal T_{\mK}$ is
\[
    \mathcal T_{\mK}^{-1}
    =
    \begin{bmatrix}
        I & -Y_{\mK} & -Y_{\mK}X_{\mK}\\
        0 & I & X_{\mK}\\
        0 & 0 & I
    \end{bmatrix}.
\]
Recall from \cref{eq:closed-loop_Lyapunov_X,eq:closed-loop_Lyapunov_Y} that $
    A_{\mathrm{cl},\mK}X_{\mK}
    +X_{\mK}A_{\mathrm{cl},\mK}^{\tr}
    +B_{\mathrm{cl},\mK}B_{\mathrm{cl},\mK}^{\tr}
    =0$ and $
    A_{\mathrm{cl},\mK}^{\tr}Y_{\mK}
    +Y_{\mK}A_{\mathrm{cl},\mK}
    +C_{\mathrm{cl},\mK}^{\tr}C_{\mathrm{cl},\mK}
    =0
$.
These identities ensure that the $(1,2)$ and $(2,3)$ blocks of the transformed $A$-matrix vanish:
\[
\mathcal{T}_\mK\begin{bmatrix}
-A_{\mathrm{cl},\mK}^{\tr}
    & C_{\mathrm{cl},\mK}^{\tr}C_{\mathrm{cl},\mK}
    & 0 \\
    0
    & A_{\mathrm{cl},\mK}
    & B_{\mathrm{cl},\mK}B_{\mathrm{cl},\mK}^{\tr} \\
    0
    & 0
    & -A_{\mathrm{cl},\mK}^{\tr}
\end{bmatrix}
\mathcal{T}_\mK^{-1}
=\begin{bmatrix}
-A_{\mathrm{cl},\mK}^{\tr}
    & 0
    & Y_{\mK}B_{\mathrm{cl},\mK}
      B_{\mathrm{cl},\mK}^{\tr}
    \\
    0
    & A_{\mathrm{cl},\mK}
    & 0
    \\
    0
    & 0
    & -A_{\mathrm{cl},\mK}^{\tr}
\end{bmatrix}.
\]
Thus, applying \cref{eq:Psi-coordinate-transformation} to
\cref{eq:Psi-cascade-realization} gives
\begin{equation*}
\def\arraystretch{1.25}
\begin{aligned}
\boldsymbol{\Psi}_{\mK}(s)
={} & 
\left[
\begin{array}{ccc|c}
    -A_{\mathrm{cl},\mK}^{\tr}
    & 0
    & Y_{\mK}B_{\mathrm{cl},\mK}
      B_{\mathrm{cl},\mK}^{\tr}
    & Y_{\mK}B_{\mathrm{cl},\mK}
      \mathfrak D_{21}^{\tr}\\
    0
    & A_{\mathrm{cl},\mK}
    & 0
    &  B_{\mathrm{cl},\mK}\mathfrak D_{21}^{\tr}
    +X_{\mK}\mathfrak C^{\tr}\\
    0
    & 0
    & -A_{\mathrm{cl},\mK}^{\tr}
    & -\mathfrak C^{\tr}\\ \hline
    -\mathfrak B^{\tr}
    & \mathfrak D_{12}^{\tr}C_{\mathrm{cl},\mK}
    +\mathfrak B^{\tr}Y_{\mK}
    & (\mathfrak D_{12}^{\tr}C_{\mathrm{cl},\mK}
    +\mathfrak B^{\tr}Y_{\mK})X_{\mK}
    & 0
\end{array}
\right]\\
={} &
\left[
\begin{array}{ccc|c}
    -A_{\mathrm{cl},\mK}^{\tr}
    & Y_{\mK}B_{\mathrm{cl},\mK}
      B_{\mathrm{cl},\mK}^{\tr}
    & 0
    & Y_{\mK}B_{\mathrm{cl},\mK}
      \mathfrak D_{21}^{\tr}\\
    0 & -A_{\mathrm{cl},\mK}^{\tr} & 0
    & -\mathfrak C^{\tr} \\
    0 & 0 & A_{\mathrm{cl},\mK}
    &  B_{\mathrm{cl},\mK}\mathfrak D_{21}^{\tr}
    +X_{\mK}\mathfrak C^{\tr}
    \\ \hline
    -\mathfrak B^{\tr}
    & (\mathfrak D_{12}^{\tr}C_{\mathrm{cl},\mK}
    +\mathfrak B^{\tr}Y_{\mK})X_{\mK}
    & \mathfrak D_{12}^{\tr}C_{\mathrm{cl},\mK}
    +\mathfrak B^{\tr}Y_{\mK} & 0
\end{array}
\right],
\end{aligned}
\end{equation*}
where we reordered the components of the state in the second equality. We group the two antistable state blocks together by defining
\begin{align*}
    \mathcal A_{\mK,-}
    \coloneq
    \begin{bmatrix}
        -A_{\mathrm{cl},\mK}^{\tr}
        &
        Y_{\mK}B_{\mathrm{cl},\mK}
        B_{\mathrm{cl},\mK}^{\tr}\\
        0
        &
        -A_{\mathrm{cl},\mK}^{\tr}
    \end{bmatrix}\!, \; 
    \mathcal B_{\mK,-}
    \coloneq
    \begin{bmatrix}
        Y_{\mK}B_{\mathrm{cl},\mK}
        \mathfrak D_{21}^{\tr}\\
        -\mathfrak C^{\tr}
    \end{bmatrix}\!, \;
    \mathcal C_{\mK,-}
    \coloneq
    \begin{bmatrix}
        -\mathfrak B^{\tr}
        &
        \!\!(\mathfrak D_{12}^{\tr}C_{\mathrm{cl},\mK}
    +\mathfrak B^{\tr}Y_{\mK})X_{\mK}
    \end{bmatrix}\!,
\end{align*}
which gives
\begin{equation}
\def\arraystretch{1.25}
\boldsymbol{\Psi}_{\mK}(s)
=
\left[
\begin{array}{cc|c}
    \mathcal A_{\mK,-}
    & 0
    & \mathcal B_{\mK,-}\\
    0
    & A_{\mathrm{cl},\mK}
    & \mathcal B_{\mathrm{cl},\mK}\mathfrak D_{21}^{\tr}
    +X_{\mK}\mathfrak C^{\tr}\\ \hline
    \mathcal C_{\mK,-}
    & \mathcal \mathfrak D_{12}^{\tr}C_{\mathrm{cl},\mK}
    +\mathfrak B^{\tr}Y_{\mK}
    & 0
\end{array}
\right].
\label{eq:Psi-separated-realization}
\end{equation}
This realization is block diagonal in its antistable and stable
state components. Consequently,
\[
    \boldsymbol{\Psi}_{\mK}(s)
    =
    \boldsymbol{\Psi}_{\mK,-}(s)
    +
    \mathbf R_{\mK}(s),
\]
where $\boldsymbol{\Psi}_{\mK,-}(s) =
    \mathcal C_{\mK,-}
    (sI-\mathcal A_{\mK,-})^{-1}
    \mathcal B_{\mK,-}$ and 
    $$\mathbf R_{\mK}(s)
    =
    \mathcal (\mathcal \mathfrak D_{12}^{\tr}C_{\mathrm{cl},\mK}
    +\mathfrak B^{\tr}Y_{\mK})
    (sI-A_{\mathrm{cl},\mK})^{-1}
    (B_{\mathrm{cl},\mK}\mathfrak D_{21}^{\tr}
    +X_{\mK}\mathfrak C^{\tr}),$$
which is the same as that defined in \Cref{theorem:global_optimality_RK}. Because $A_{\mathrm{cl},\mK}$ is Hurwitz, $\mathbf R_{\mK}$ is strictly proper and stable. The eigenvalues of $\mathcal A_{\mK,-}$ are precisely those of $-A_{\mathrm{cl},\mK}^{\tr}$, each repeated twice, and therefore have positive real parts. Hence $\boldsymbol{\Psi}_{\mK,-}$ is strictly proper and antistable. This proves the claimed stable--antistable decomposition.
\end{proof}

\section{Hessian Computation and Proof of \Cref{lemma:hessian_augmented_scalar}}
\label{appendix:proof_hessian_augmented_scalar}
The proof utilizes the following lemma on the Hessian of the LQG objective function.

\begin{lemma}[{\cite[Lemma 9]{tang2023analysis}}]
\label{lemma:hessian_general}
Fix $q\in\mathbb{Z}_{>0}$ such that $\mathcal{C}_q\neq \varnothing$. Let $\mK=\begin{bmatrix}
0 & C_\mK \\ B_\mK & A_\mK
\end{bmatrix}\in\mathcal{C}_q$. Then for any $\Delta=\begin{bmatrix}
0 & \Delta_{C_\mK} \\
\Delta_{B_\mK} & \Delta_{A_\mK}
\end{bmatrix}\in\mathcal{V}_q$, we have
\[
\begin{aligned}
& D^2 J_q(\mK)[\Delta,\Delta] \\
={} & 2\operatorname{tr}\!\Bigg(
2\begin{bmatrix}
B & 0_{n\times q} \\ 0_{q\times m} & I_q
\end{bmatrix}\Delta\begin{bmatrix}
C & 0_{\ell\times q} \\ 0_{q\times n} & I_q
\end{bmatrix}
X_{\mK,\Delta}' Y_{\mK} + 2
\begin{bmatrix}
0_{n\times m} & 0_{n\times q} \\ C_\mK^\tran R & 0_{q\times q}
\end{bmatrix}\Delta
\begin{bmatrix}
0_{\ell\times n} & 0_{\ell\times q} \\
0_{q\times n} & I_q
\end{bmatrix}X'_{\mK,\Delta} \\
&\qquad\quad+
\begin{bmatrix}
0_{n\times n} & 0_{n\times q} \\
0_{q\times n} & \Delta_{B_\mK} V \Delta_{B_\mK}^\tran
\end{bmatrix} Y_\mK
+\begin{bmatrix}
0_{n\times n} & 0_{n\times q} \\
0_{q\times n} & \Delta_{C_\mK}^\tran R \Delta_{C_\mK}
\end{bmatrix} X_\mK
\Bigg),
\end{aligned}
\]
where $X_{\mK}$ and $Y_{\mK}$ are the solutions to the Lyapunov equation~\cref{eq:closed-loop_Lyapunov_X,eq:closed-loop_Lyapunov_Y}, and $X_{\mK,\Delta}'$ is the solution to the following Lyapunov equation
\[
A_{\mathrm{cl},\mK} X_{\mK,\Delta}'
+X_{\mK,\Delta}' A_{\mathrm{cl},\mK}^\tran
+M_1(X_\mK,\Delta)
+M_1(X_\mK,\Delta)^\tran= 0,
\]
where
\[
\begin{aligned}
M_1(X_\mK,\Delta)\coloneq{} &
\begin{bmatrix}
B & 0_{n\times q} \\ 0_{q\times m} & I_q
\end{bmatrix}\Delta\begin{bmatrix}
C & 0_{\ell\times q} \\ 0_{q\times n} & I_q
\end{bmatrix} X_\mK 
+\begin{bmatrix}
0_{n\times m} & 0_{n\times q} \\
0_{q\times m} & I_q
\end{bmatrix}\Delta\begin{bmatrix}
0_{\ell\times n} & VB_\mK^\tran \\
0_{q\times n} & 0_{q\times q}
\end{bmatrix}.
\end{aligned}
\]
\end{lemma}

We next prove \cref{lemma:hessian_augmented_scalar}. Denote
$
\overline{\mK} = \mK\oplus\lambda I_p
$
for notational simplicity. We then have
\[
A_{\mathrm{cl},\overline\mK}
=\begin{bmatrix}
A_{\mathrm{cl},\mK} & 0_{(n+q)\times p} \\
0_{p\times(n+q)} & \lambda I_p
\end{bmatrix},
\quad
B_{\mathrm{cl},\overline\mK}
=\begin{bmatrix}
B_{\mathrm{cl},\mK} \\
0_{p\times(n+\ell)}
\end{bmatrix},
\quad
C_{\mathrm{cl},\overline\mK}
=\begin{bmatrix}
C_{\mathrm{cl},\mK} & 0_{(n+m)\times p}
\end{bmatrix}.
\]
The Lyapunov equations \cref{eq:closed-loop_Lyapunov_X,eq:closed-loop_Lyapunov_Y} for the policy $\overline{\mK}$ now become
\begin{align}
\begin{bmatrix}
A_{\mathrm{cl},\mK} & 0 \\
0 & \lambda I_p
\end{bmatrix}X_{\overline\mK}
+X_{\overline\mK}
\begin{bmatrix}
A_{\mathrm{cl},\mK} & 0 \\
0 & \lambda I_p
\end{bmatrix}^\tran
+\begin{bmatrix}
B_{\mathrm{cl},\mK}B_{\mathrm{cl},\mK}^\tran & 0 \\ 0 & 0
\end{bmatrix} ={} & 0, 
\nonumber \\
\begin{bmatrix}
A_{\mathrm{cl},\mK} & 0 \\
0 & \lambda I_p
\end{bmatrix}^\tran Y_{\overline\mK}
+Y_{\overline\mK}
\begin{bmatrix}
A_{\mathrm{cl},\mK} & 0 \\
0 & \lambda I_p
\end{bmatrix}
+\begin{bmatrix}
C_{\mathrm{cl},\mK}^\tran C_{\mathrm{cl},\mK} & 0 \\ 0 & 0
\end{bmatrix} ={} & 0, \nonumber
\end{align}
Since $\lambda<0$ and $\mK\in\mathcal{C}_q$, the matrix $A_{\mathrm{cl},\overline\mK}$ is Hurwitz stable, and thus the above two Lyapunov equations have unique solutions. It is then straightforward to check that their solutions are
\[
X_{\overline\mK} = \begin{bmatrix}
X_\mK & 0_{(n+q)\times p} \\ 0_{p\times(n+q)} & 0_{p\times p}
\end{bmatrix},
\qquad 
Y_{\overline\mK} = \begin{bmatrix}
Y_\mK & 0_{(n+q)\times p} \\ 0_{p\times(n+q)} & 0_{p\times p}
\end{bmatrix},
\]
where $X_\mK$ and $Y_\mK$ are the solutions to the Lyapunov equations~\cref{eq:closed-loop_Lyapunov_X,eq:closed-loop_Lyapunov_Y} for the policy $\mK$.

Now fix $\Delta=\begin{bmatrix}
0 & \Delta_{12} \\ \Delta_{21} & 0
\end{bmatrix}\in\mathcal{M}_{q,p}$, and partition $\Delta$ as follows:
\[
\Delta = \begin{bmatrix}
0_{(m+q)\times(\ell+q)} & \Delta_{12} \\
\Delta_{21} & 0_{p\times p}
\end{bmatrix}
=
\begin{bmatrix}
0_{m\times\ell} & 0_{q\times\ell} & \Delta_{C} \\
0_{q\times\ell} & 0_{q\times q} & \Delta_{A,1} \\
\Delta_B & \Delta_{A,2} & 0_{p\times p}
\end{bmatrix},
\]
Applying \cref{lemma:hessian_general} to the policy $\overline\mK$ leads to
\begin{equation}
\label{eq:hessian_computation_step1}
\begin{aligned}
& D^2 J_{q+p}(\overline\mK)[\Delta,\Delta] \\
={} & 4\operatorname{tr}\!\left(
\begin{bmatrix}
B & 0_{n\times(q+p)} \\ 0_{(q+p)\times m} & I_{q+p}
\end{bmatrix}\Delta\begin{bmatrix}
C & 0_{\ell\times (q+p)} \\ 0_{(q+p)\times n} & I_{q+p}
\end{bmatrix}
X_{\overline\mK,\Delta}' Y_{\overline\mK}\right) \\
&+ 4\operatorname{tr}\!\left(
\begin{bmatrix}
0_{n\times m} & 0_{n\times (q+p)} \\ \begin{bmatrix}
C_\mK^\tran \\
0_{p\times m}
\end{bmatrix} R & 0_{(q+p)\times (q+p)}
\end{bmatrix}\Delta
\begin{bmatrix}
0_{\ell\times n} & 0_{\ell\times (q+p)} \\
0_{(q+p)\times n} & I_{q+p}
\end{bmatrix}X'_{\mK,\Delta}\right) \\
&+2\operatorname{tr}\!\left(
\begin{bmatrix}
0_{n\times n} & 0_{n\times (q+p)} \\
0_{(q+p)\times n} & \begin{bmatrix}
0_{q\times \ell} \\
\Delta_B
\end{bmatrix}
V \begin{bmatrix}
0_{\ell\times q} & \Delta_B^\tran
\end{bmatrix}
\end{bmatrix} Y_\mK
+\begin{bmatrix}
0_{n\times n} & 0_{n\times (q+p)} \\
0_{(q+p)\times n} & 
\begin{bmatrix}
0_{q\times m} \\
\Delta_C^\tran
\end{bmatrix} R \begin{bmatrix}
0_{m\times q} & \Delta_C
\end{bmatrix}
\end{bmatrix} X_\mK
\right),
\end{aligned}
\end{equation}
where $X_{\overline\mK,\Delta}'$ is the solution to the Lyapunov equation
\begin{equation}
\label{eq:hessian_Lyapunov_Xprime}
\begin{bmatrix}
A_{\mathrm{cl},\mK} & 0_{(n+q)\times p} \\
0_{p\times(n+q)} & \lambda I_p
\end{bmatrix} X_{\overline\mK,\Delta}'
+X_{\overline\mK,\Delta}' \begin{bmatrix}
A_{\mathrm{cl},\mK} & 0_{(n+q)\times p} \\
0_{p\times(n+q)} & \lambda I_p
\end{bmatrix}^\tran
+M_1(X_{\overline\mK},\Delta)
+M_1(X_{\overline\mK},\Delta)^\tran
= 0,
\end{equation}
with
\[
\begin{aligned}
M_1(X_{\overline\mK},\Delta)={} &
\begin{bmatrix}
B & 0_{n\times (q+p)} \\ 0_{(q+p)\times m} & I_{q+p}
\end{bmatrix}\Delta\begin{bmatrix}
C & 0_{\ell\times (q+p)} \\ 0_{(q+p)\times n} & I_{q+p}
\end{bmatrix} X_{\overline\mK} \\
&
+\begin{bmatrix}
0_{n\times m} & 0_{n\times (q+p)} \\
0_{(q+p)\times m} & I_{q+p}
\end{bmatrix}\Delta\begin{bmatrix}
0_{\ell\times n} & V\begin{bmatrix} B_\mK^\tran & 0_{\ell\times p}\end{bmatrix} \\
0_{(q+p)\times n} & 0_{(q+p)\times (q+p)}
\end{bmatrix}.
\end{aligned}
\]

In order to simplify the expression for $D^2 J_{q+p}(\overline\mK)[\Delta,\Delta]$, we first note that
\[
\begin{bmatrix}
0_{n\times n} & 0_{n\times (q+p)} \\
0_{(q+p)\times n} & \begin{bmatrix}
0_{q\times \ell} \\
\Delta_B
\end{bmatrix}
V \begin{bmatrix}
0_{\ell\times q} & \Delta_B^\tran
\end{bmatrix}
\end{bmatrix} Y_\mK
=\begin{bmatrix}
0_{(n+q)\times(n+q)} & 0_{(n+q)\times p} \\
0_{p\times(n+q)} & \Delta_B V\Delta_B^\tran
\end{bmatrix}
\begin{bmatrix}
Y_\mK & 0_{(n+q)\times p} \\ 0_{p\times(n+q)} & 0_{p\times p}
\end{bmatrix} =0,
\]
and similarly
\[
\begin{bmatrix}
0_{n\times n} & 0_{n\times (q+p)} \\
0_{(q+p)\times n} & 
\begin{bmatrix}
0_{q\times m} \\
\Delta_C^\tran
\end{bmatrix} R \begin{bmatrix}
0_{m\times q} & \Delta_C
\end{bmatrix}
\end{bmatrix} X_\mK = 0.
\]
Therefore the terms in the last line of~\cref{eq:hessian_computation_step1} all vanish. Next, we denote
\[
\mathfrak{B} = \begin{bmatrix}
B & 0_{n\times q} \\ 0_{q\times m} & I_{q}
\end{bmatrix},
\qquad
\mathfrak{C} = \begin{bmatrix}
C & 0_{\ell\times q} \\ 0_{q\times n} & I_q
\end{bmatrix},
\]
so that
\begin{equation}
\label{eq:hessian_computation_BdeltaC}
\begin{aligned}
& \begin{bmatrix}
B & 0_{n\times (q+p)} \\ 0_{(q+p)\times m} & I_{q+p}
\end{bmatrix}\Delta\begin{bmatrix}
C & 0_{\ell\times (q+p)} \\ 0_{(q+p)\times n} & I_{q+p}
\end{bmatrix} \\
={} &
\begin{bmatrix}
\mathfrak{B} & 0_{(n+q)\times p} \\
0_{p\times (m+q)} & I_p
\end{bmatrix}
\begin{bmatrix}
0_{(m+q)\times(\ell+q)} & \Delta_{12} \\
\Delta_{21} & 0_{p\times p}
\end{bmatrix}
\begin{bmatrix}
\mathfrak{C} & 0_{(\ell+q)\times p} \\
0_{p\times (n+q)} & I_p
\end{bmatrix} \\
={} &
\begin{bmatrix}
0_{(n+q)\times(n+q)} & \mathfrak{B}\Delta_{12} \\
\Delta_{21}\mathfrak{C} & 0_{q\times q}
\end{bmatrix}.
\end{aligned}
\end{equation}
By plugging~\cref{eq:hessian_computation_BdeltaC} back to~\cref{eq:hessian_computation_step1} and introducing the matrix
\[
\Pi_{C_\mK} = \begin{bmatrix}
0_{n\times m} & 0_{n\times q} \\
C_\mK^\tran R & 0_{q\times q}
\end{bmatrix},
\]
we get
\begin{equation}
\label{eq:hessian_computation_step2}
\begin{aligned}
& D^2 J_{q+p}(\overline\mK)[\Delta,\Delta] \\
={} &
4\operatorname{tr}
\left(
\begin{bmatrix}
Y_\mK & 0_{(n+q)\times p} \\
0_{p\times(n+q)} & 0_{p\times p}
\end{bmatrix}
\begin{bmatrix}
0_{(n+q)\times(n+q)} & \mathfrak{B}\Delta_{12} \\
\Delta_{21}\mathfrak{C} & 0_{q\times q}
\end{bmatrix}X_{\overline\mK,\Delta}'
\right) \\
&
+4\operatorname{tr}\!\left(
\begin{bmatrix}
\Pi_{C_\mK} & 0_{(n+q)\times p} \\ 
0_{p\times(m+q)} & 0_{p\times p}
\end{bmatrix}\begin{bmatrix}
0_{(m+q)\times(\ell+q)} & \Delta_{12} \\
\Delta_{21} & 0_{p\times p}
\end{bmatrix}
\begin{bmatrix}
\ast & 0_{(\ell+q)\times p} \\
0_{p\times (n+q)} & I_{p}
\end{bmatrix}X'_{\overline\mK,\Delta}\right) \\
={} &
4\operatorname{tr}
\!\left(
X'_{\overline\mK,\Delta}\begin{bmatrix}
0_{(n+q)\times (n+q)} &
(Y_\mK\mathfrak{B}+\Pi_{C_\mK})\Delta_{12} \\
0_{p\times(n+q)} & 0_{p\times p}
\end{bmatrix}
\right),
\end{aligned}
\end{equation}
where we used $\operatorname{tr}(M_1M_2)=\operatorname{tr}(M_2M_1)$ in the calculation, and $\ast$ represents a block whose value is irrelevant to the final result.

Next we solve for $X_{\overline\mK,\Delta}'$. Note that by~\cref{eq:hessian_computation_step2}, we only need to find the first $(n+q)$ columns of $X_{\overline\mK,\Delta}'$. Denote
\[
\Pi_{B_\mK} = \begin{bmatrix}
0_{\ell\times n} & VB_\mK^\tran \\
0_{q\times n} & 0_{q\times q}
\end{bmatrix}.
\]
Then
\[
\begin{aligned}
M_1(X_{\overline\mK},\Delta)
={} &
\begin{bmatrix}
0_{(n+q)\times(n+q)} & \mathfrak{B}\Delta_{12} \\
\Delta_{21}\mathfrak{C} & 0_{q\times q}
\end{bmatrix}\begin{bmatrix}
X_\mK & 0_{(n+q)\times p} \\
0_{p\times(n+q)} & 0_{p\times p}
\end{bmatrix} \\
&+
\begin{bmatrix}
\ast & 0_{(n+q)\times p} \\
0_{p\times (m+q)} & I_p
\end{bmatrix}\begin{bmatrix}
0_{(m+q)\times(\ell+q)} & \Delta_{12} \\
\Delta_{21} & 0_{p\times p}
\end{bmatrix}
\begin{bmatrix}
\Pi_{B_\mK} & 0_{(\ell+q)\times p} \\
0_{p\times(n+q)} & 0_{p\times p}
\end{bmatrix} \\
={} &
\begin{bmatrix}
0_{(n+q)\times(n+q)} & 0_{(n+q)\times p} \\
\Delta_{21}(\mathcal{C}X_\mK+\Pi_{B_\mK}) & 0_{q\times q}
\end{bmatrix}.
\end{aligned}
\]
The Lyapunov equation~\cref{eq:hessian_Lyapunov_Xprime} for $X_{\overline\mK,\Delta}'$ is then
\[
\begin{aligned}
0 ={} & \begin{bmatrix}
A_{\mathrm{cl},\mK} & 0_{(n+q)\times p} \\
0_{p\times(n+q)} & \lambda I_p
\end{bmatrix} X_{\overline\mK,\Delta}'
+X_{\overline\mK,\Delta}' \begin{bmatrix}
A_{\mathrm{cl},\mK} & 0_{(n+q)\times p} \\
0_{p\times(n+q)} & \lambda I_p
\end{bmatrix}^\tran \\
& +
\begin{bmatrix}
0_{(n+q)\times(n+q)} & (\mathcal{C}X_\mK + \Pi_{B_\mK})^\tran\Delta_{21}^\tran \\
\Delta_{21}(\mathcal{C}X_\mK + \Pi_{B_\mK}) & 0_{q\times q}
\end{bmatrix},
\end{aligned}
\]
The solution is then
\[
\begin{aligned}
X_{\overline\mK,\Delta}'
={} & \!\int_0^{+\infty}
\!
\begin{bmatrix}
\exp(tA_{\mathrm{cl},\mK}) & \!\!\!0_{(n+q)\times p} \\
0_{p\times(n+q)} & \!\!\!e^{\lambda t} I_p
\end{bmatrix}\!\!
\begin{bmatrix}
0_{(n+q)\times(n+q)} & \ast \\
\Delta_{21}(\mathcal{C}X_\mK\!+\!\Pi_{B_\mK}) &  \ast
\end{bmatrix}\!\!
\begin{bmatrix}
\exp(tA_{\mathrm{cl},\mK}^\tran) & \!\!\! 0_{(n+q)\times p} \\
0_{p\times(n+q)} & \!\!\! e^{\lambda t} I_p
\end{bmatrix}\!\mathrm{d}t \\
={} &
\int_0^{+\infty}
\begin{bmatrix}
0_{(n+q)\times(n+q)} & \ast \\
\Delta_{21}(\mathcal{C}X_\mK\!+\!\Pi_{B_\mK})
\exp\!\left[t(\lambda I\!+\!A_{\mathrm{cl},\mK})^\tran\right] & \ast
\end{bmatrix}\mathrm{d}t \\
={} &
\begin{bmatrix}
0_{(n+q)\times(n+q)} & 
\ast \\
\Delta_{21}(\mathcal{C}X_\mK\!+\!\Pi_{B_\mK})
(-\lambda I\!-\!A_{\mathrm{cl},\mK})^{-\tran}
& \ast
\end{bmatrix},
\end{aligned}
\]
where we used the identity $\int_0^{+\infty}\exp(tM)\,\mathrm{d}t=-M^{-1}$ when $M$ is Hurwitz stable; the matrix $\lambda I+A_{\mathrm{cl},\mK}$ is Hurwitz stable because $A_{\mathrm{cl},\mK}$ is Hurwitz stable and $\lambda<0$. By plugging the above expression into~\cref{eq:hessian_computation_step2}, we get
\[
D^2 J_{q+p}(\overline\mK)[\Delta,\Delta] 
=4\operatorname{tr}
\begin{bmatrix}
\Delta_{21}(\mathcal{C}X_\mK+\Pi_{B_\mK})
(-\lambda I-A_{\mathrm{cl},\mK})^{-\tran}(Y_\mK\mathfrak{B}+\Pi_{C_\mK})\Delta_{12}
\end{bmatrix}.
\]
The proof is now complete.

\end{document}